\documentclass[english,12pt]{amsart}
\usepackage[T1]{fontenc}
\usepackage[latin9]{inputenc}
\usepackage[letterpaper]{geometry}
\usepackage{varioref}
\usepackage{amsthm}
\usepackage{amstext}
\usepackage{amssymb}
\usepackage{graphicx}
\usepackage{esint}

\makeatletter
\numberwithin{equation}{section}
\numberwithin{figure}{section}
  \theoremstyle{plain}
  \newtheorem*{thm*}{\protect\theoremname}
\theoremstyle{plain}
\newtheorem{thm}{\protect\theoremname}
  \theoremstyle{plain}
  \newtheorem{lem}[thm]{\protect\lemmaname}
  \theoremstyle{remark}
  \newtheorem{rem}[thm]{\protect\remarkname}

\makeatother

\usepackage{babel}
  \providecommand{\lemmaname}{Lemma}
  \providecommand{\remarkname}{Remark}
  \providecommand{\theoremname}{Theorem}
\providecommand{\theoremname}{Theorem}

\begin{document}

\title{Continuity of the asymptotics of expected zeros of fewnomials}

\author{Timothy Tran}

\date{11/25/2013}
\begin{abstract}
In ``Random complex fewnomials, I,'' B. Shiffman and S. Zelditch
determine the limiting formula as $N\to\infty$ of the (normalized)
expected distribution of complex zeros of a system of $k$ random
$n$-nomials in $m$ variables $(k\leq m)$ where the coefficients
are taken from the $\mathrm{SU}(m+1)$ ensemble and the spectra are
chosen uniformly at random. We recall their result and show the limiting
formula is a $(k,k)$-form with continuous coefficients.
\end{abstract}
\maketitle
\tableofcontents{}

\section{Introduction \label{sec:Introduction}}

A. G. Kushnirenko conjectured in the late 1970s that the number of
real roots of a general system of $m$ polynomials in $m$ variables
can be bounded above by a function of the number of nonzero terms
which appear \cite{Kh}. Equivalently, Kushnirenko's conjecture states
that the complexity of a polynomial system can be measured by the
number of nonzero terms among the polynomials in the system instead
of the degrees of the polynomials in the system. As it turns out,
A. Khovanskii's work on fewnomial systems in the 1980s showed this
conjecture is true.

Besides answering Kushnirenko's conjecture, a theorem by Khovanskii
(\ref{eq:Khovanskii}) can be interpreted as obtaining a bound on
the deviation from the average number of zeros in an angular sector.
Obtaining a corollary (\ref{eq:Khovanskii2}) to this theorem, Khovanskii
finds the number of nondegenerate roots of a polynomial system $P=(P_{1}=\dotsb=P_{m}=0)$
lying in the positive orthant of $\mathbb{R}^{m}$ is at most 
\[
2^{n(n-1)/2}(m+1)^{n},
\]
where $n$ is the number of distinct monomials that appear with nonzero
coefficient in at least one of the polynomials \cite{Kh}. However,
this bound is not sharp, and even an improvement of the bound to 
\[
\frac{e^{2}+3}{4}2^{l(l-1)/2}m^{l},
\]
where $m+l+1$ is the number of monomials, by Bihan and Sottile \cite{BS2}
is only asymptotically sharp \cite{BRS}. Consequently, rather than
looking for a sharp bound on the number of zeros of a polynomial system
$P$, we wish to consider the typical behavior of zeros to complex
fewnomial systems and are motivated to study the distribution of zeros
probabilistically.

With this in mind, let us introduce some notation and state two theorems
by B. Shiffman and S. Zelditch which give the limit of the expected
distribution of complex zeros of a system of $k$ random $n$-nomials
in $m$ variables, where the first theorem considers a fixed choice
of exponents which are then dilated, while the second theorem considers
choosing among all possible choices of exponents uniformly.

We will begin by defining the $\mathrm{SU}(m+1)$ ensemble. Denote
the space of polynomials of degree at most $N$ by $\mathrm{Poly}(N)$
and place on it the inner product $h$ given by
\[
\left\langle f,\bar{g}\right\rangle =\frac{1}{m!}\int_{\mathbb{C}^{m}}\frac{f(z)\overline{g(z)}}{\left(1+\left\Vert z\right\Vert ^{2}\right)^{N}}\omega_{FS}^{m}(z),\quad\text{for }f,g\in\mathrm{Poly}(N),
\]
where $\omega_{FS}=\frac{i}{2\pi}\partial\bar{\partial}\log(1+\left\Vert z\right\Vert ^{2})$
is the Fubini-Study Kähler form on $\mathbb{C}^{m}\subset\mathbb{CP}^{m}$.
Then since $h$ is invariant under the torus action, the monomials
$\chi_{\alpha}$ determine an orthogonal basis for $\mathrm{Poly}(N)$
with respect to $h$ (true for any inner product invariant under the
torus action). Denoting the normalization of $\chi_{\alpha}$ with
respect to $h$ by $m_{\alpha}$, we obtain the expression of a polynomial
$p\in\mathrm{Poly}(N)$ as 
\[
p=\sum_{\alpha\in\mathbb{N}^{m},\left|\alpha\right|\leq N}c_{\alpha}m_{\alpha},
\]
where $\left|\alpha\right|=\alpha_{1}+\dots+\alpha_{m}.$ Now define
the induced Gaussian (probability) measure $\gamma_{N}$ on $\mathrm{Poly}(N)$
as the measure such that the coefficients $c_{\alpha}$ are independent
complex normal random variables. Having placed a measure on $\mathrm{Poly}(N)$,
we define the $\mathrm{SU}(m+1)$ ensemble as the spaces $\mathrm{Poly}(N)$
together with their corresponding measures $\gamma_{N}$.

Next, we consider subspaces of $\mathrm{Poly}(N)$ and their corresponding
conditional probability measures. Given $p\in\mathrm{Poly}(N)$, we
define the spectrum of $p$ to be the set $S_{p}=\{\alpha\in\mathbb{N}^{m}\mid c_{\alpha}\neq0\}$.
Then choose a set of exponents $S\subset\{\alpha\in\mathbb{N}^{m}\mid\left|\alpha\right|\leq N\}$
and consider the subspace $\mathrm{Poly}(S)$ of $\mathrm{Poly}(N)$
consisting of polynomials $p$ whose spectrum $S_{p}$ is contained
in $S$. Restricting $\gamma_{N}$ to this subspace, we denote the
conditional probability by $\gamma_{N\mid S}$. We can also consider
all possible sets $S$ for which the cardinality of the set $S$ is
some fixed $n\in\mathbb{N}$. In this situation, we shall consider
the measure $\gamma_{N\mid n}$ such that the restriction of $\gamma_{N\mid n}$
to $\mathrm{Poly}(S)$ is $\frac{1}{\left|C(N,n)\right|}\gamma_{N\mid S}$,
where $C(N,n)$ is collection of sets $S\subset\{\alpha\in\mathbb{N}^{m}\mid\left|\alpha\right|\leq N\}$
with cardinality $n$ and $\left|C(N,n)\right|$ is the cardinality
of $C(N,n)$. In regard to terminology, we say a polynomial $p$ is
an $n$-nomial if the cardinality of $S_{p}$ is at most $n$.

Next, let us fix a subspace $\mathcal{S}$ in $\mathrm{Poly}(N)$.
For polynomials $P_{1},\dots,P_{k}\in\mathcal{S}$ and $k\leq m$,
we consider the set of zeros $Z_{P_{1},\dots,P_{k}}=\{z\in\left(\mathbb{C}^{\ast}\right)^{m}\mid P_{1}(z)=\cdots=P_{k}(z)=0\}$
and the associated current of integration $[Z_{P_{1},\dots,P_{k}}]\in\mathcal{D}^{\prime k,k}((\mathbb{C}^{\ast})^{m})$
given by 
\[
([Z_{P_{1},\dots,P_{k}}],\psi)=\int_{Z_{p_{1}\cdots P_{k}}}\psi,\quad\text{for }\psi\in\mathcal{D}^{m-k,m-k}((\mathbb{C}^{\ast})^{m}).
\]
Then the expected zero current, if it exists, is a $(k,k)$-current
$\mathbb{E}_{\mathcal{S}}[Z_{P_{1},\dots,P_{k}}]$ such that
\[
(\mathbb{E}_{\mathcal{S}}[Z_{P_{1},\dots,P_{k}}],\psi)=\mathbb{E}_{\mathcal{S}}([Z_{P_{1},\dots,P_{k}}],\psi),\quad\psi\in\mathcal{D}^{m-k,m-k}((\mathbb{C}^{\ast})^{m}).
\]

Finally, we let $\Delta$ denote the unit simplex in $\mathbb{R}^{m}$,
write $\left|\lambda\right|=\lambda_{1}+\dotsb+\lambda_{m}$ for $\lambda\in\mathbb{R}^{m}$,
and write $z=[\exp(\frac{1}{2}\rho_{1}+i\theta_{1}),\dots,\exp(\frac{1}{2}\rho_{m}+i\theta_{m})]$
for $z\in(\mathbb{C}^{\ast})^{m}$. Having provided the fundamental
background and notation, we first look at the theorem for a fixed
spectrum $S$ and how the expected distribution of zeros behaves as
$S$ is dilated by a factor $N$.
\begin{thm*}[Shiffman and Zelditch, Theorem 1.2 \cite{SZ}]
 Let $S=\{\lambda^{1},\dots,\lambda^{n}\}$ be a fixed spectrum consisting
of $n$ lattice points in $p\Delta$. For random $m$-tuples $(P_{1}^{N},\dots,P_{m}^{N})$
of $n$-nomials in $\mathrm{Poly}(NS)$, with coefficients chosen
from the $\mathrm{SU}(m+1)$ ensembles of degree $Np$, the expected
distribution of zeros in $(\mathbb{C}^{\ast})^{m}$ has the asymptotics
\[
N^{-m}\mathbb{E}_{Np\mid NS}[Z_{P_{1}^{N},\dots,P_{m}^{N}}]\to p^{m}\left(\frac{i}{2\pi}\partial\bar{\partial}\max_{\lambda\in S}\left[\left\langle \rho,\lambda\right\rangle -\bigl\langle\hat{\lambda},\log\hat{\lambda}\vphantom{}^{p}\bigr\rangle\right]\right)^{k}.
\]
Here, $\hat{\lambda}\vphantom{}^{p}=(p-\left|\lambda\right|,\lambda_{1},\dots,\lambda_{m})$
and $\log\hat{\lambda}\vphantom{}^{p}=(\log(p-\left|\lambda\right|),\log\lambda_{1},\dots,\log\lambda_{m})$. 
\end{thm*}
Looking at the potential function $\rho\mapsto\max_{\lambda\in S}[\left\langle \rho,\lambda\right\rangle -\bigl\langle\hat{\lambda},\log\hat{\lambda}\vphantom{}^{p}\bigr\rangle]$
of the limit, we note that it is a piecewise linear function, and
so the the expected limit distribution is a singular measure supported
along where the functions meet. For the case when $k=m$, the measure
is supported along the $0$-dimensional corner set. Keeping this remark
in the back of our mind, we look at the theorem where the spectra
are chosen uniformly at random.
\begin{thm*}[Shiffman and Zelditch, Theorem 1.5 \cite{SZ}]
 Let $1\leq k\leq m$, and let $(P_{1},\dots,P_{k})$ be a random
system of $n$-nomials of degree $N$, where the spectra $S_{j}$
are chosen uniformly at random from the simplex $N\Delta$ and the
coefficients are chosen from the $\text{SU}(m+1)$ ensemble. Then
the expected zero current in $(\mathbb{C}^{\ast})^{m}$ has the asymptotics
\[
N^{-k}\mathbb{E}_{N,n}[Z_{P_{1}^{N},\dots,P_{k}^{N}}]\to\left(\frac{i}{2\pi}\partial\bar{\partial}\int_{\Delta^{n}}\max_{j=1,\dots,n}\left[\bigl\langle\rho,\lambda^{j}\bigr\rangle-\bigl\langle\widehat{\lambda^{j}},\log\widehat{\lambda^{j}}\bigr\rangle\right]d\lambda^{1}\cdots d\lambda^{n}\right)^{k}.
\]
Here, $\hat{\lambda}=(1-\left|\lambda\right|,\lambda_{1},\dotsc,\lambda_{m})$,
\textup{$\log\hat{\lambda}=(\log(1-\left|\lambda\right|),\log\lambda_{1},\dots,\log\lambda_{m})$,
and} $d\lambda=m!d\lambda_{1}\dotsm d\lambda_{m}$.
\end{thm*}
The first observation we make is the appearance of the piecewise linear
function $(\rho,\lambda)\mapsto\max_{j=1,\dots,n}[\bigl\langle\rho,\lambda^{j}\bigr\rangle-\bigl\langle\widehat{\lambda^{j}},\log\widehat{\lambda^{j}}\bigr\rangle]$
in the integrand of the potential function of the limit. This makes
sense, at least intuitively, because choosing the spectra uniformly
at random should result in averaging over the potential functions
of the individual spectra. However, because the expected limit distribution
for the dilation of fixed spectra is singular, it is unclear whether
their average should be singular or nonsingular. Resolving this uncertainty,
we prove the following theorem:
\begin{thm}
\label{thm:maintheorem}Define the function
\begin{gather*}
F_{n}:\mathbb{R}^{m}\to\mathbb{R}\\
\rho\mapsto\int_{\Delta^{n}}\max_{j=1,\dots,n}\left[\bigl\langle\rho,\lambda^{j}\bigr\rangle-\bigl\langle\widehat{\lambda^{j}},\log\widehat{\lambda^{j}}\bigr\rangle\right]d\lambda^{1}\dotsm d\lambda^{n}.
\end{gather*}
Then $F_{n}$ is $\mathcal{C}^{2}(\mathbb{R}^{m})$ and thus the distribution
\begin{equation}
\left(\frac{i}{2\pi}\partial\bar{\partial}\int_{\Delta^{n}}\max_{j=1,\dots,n}\left[\bigl\langle\rho,\lambda^{j}\bigr\rangle-\bigl\langle\widehat{\lambda^{j}},\log\widehat{\lambda^{j}}\bigr\rangle\right]d\lambda^{1}\cdots d\lambda^{n}\right)^{k}\label{eq:distribution}
\end{equation}
is a $(k,k)$-form with continuous coefficients.

In particular, when $k=m$, the distribution in (\ref{eq:distribution})
is given by 
\[
m!\det\left(\frac{1}{2\pi}\frac{\partial^{2}}{\partial\rho_{i}\partial\rho_{j}}F_{n}\right)d\rho d\theta,
\]
which has a continuous density with respect to the Lebesgue measure
on $(\mathbb{C}^{\ast})^{m}$. 
\end{thm}
Having stated the goal of our paper, let us return to the beginning
of the introduction and readdress Kushnirenko's conjecture. When $m=1$,
we can answer the conjecture using Descartes' rule of signs. Ordering
the terms of a polynomial by increasing exponent, the rule states
that the number of positive roots is equal to the number of sign changes
between consecutive nonzero coefficients or less than it by a multiple
of two. Thus, a univariate polynomial with $n$ nonzero terms has
at most $n-1$ positive roots, regardless of the polynomial's degree.
For $m>1$, we consider the real polynomial system $P=(P_{1}=\cdots=P_{m}=0)$
on $(\mathbb{C}^{\ast})^{m}$ and let $\Delta_{j}$ denote the Newton
polytope of $P_{j}$, that is, the convex hull of the exponents appearing
nontrivially in $P_{j}$. Then let $V(\Delta_{1},\dotsc,\Delta_{m})$
denote the mixed volume of $\Delta_{1},\dotsc,\Delta_{m}$. By the
Bernstein-Kushnirenko Theorem, the total number of solutions to a
general system $P$ is $m!V(\Delta_{1},\dotsc,\Delta_{m})$. Next
consider the real $m$-torus $\mathbb{T}^{m}=\left(\mathbb{R}/[0,2\pi]\right)^{m}$.
Then $(\mathbb{C}^{\ast})^{m}$ can be identified with $\mathbb{R}^{m}\times\mathbb{T}^{m}$
under the map $z=\left[\exp(\frac{1}{2}\rho_{1}+i\theta_{1}),\dots,\exp(\frac{1}{2}\rho_{m}+i\theta_{m})\right]\mapsto(\rho,\theta)$
and we call $\theta$ the argument of $z$. Fixing $U\subset\mathbb{T}^{m}$,
let $N(P,U)$ be the number of zeros of $P$ with arguments lying
in $U$ and let $S(P,U)=m!V(\Delta_{1},\dotsc,\Delta_{m})\mathrm{Vol}(U)/\mathrm{Vol}(\mathbb{T}^{m})$
be the average number of zeros of $P$ in $U$. Let $\Delta^{\ast}\subset\mathbb{R}^{m}$
denote the intersection of the sets $\{y\in\mathbb{R}^{m}\mid\left\langle \alpha,y\right\rangle <\pi\}$
as $\alpha$ ranges over the exponents appearing in the spectra $\Delta_{i}$.
Then let $\overline{\Delta}^{\ast}$ denote the image of the region
$\Delta^{\ast}$ under the quotient homomorphism $\phi:\mathbb{R}^{m}\to\mathbb{T}^{m}$
and $\Pi(\overline{\Delta}^{\ast},U)$ be the minimal number of translates
of $\overline{\Delta}^{\ast}$ needed to cover the boundary of $U$
. Khovanskii's result on real fewnomial systems says the following:
\begin{thm*}[Khovanskii, Theorem 2 \cite{Kh}]
 There exists an explicit function $\varphi(n,m)$ of two natural
number variables $(n,m)$ such that for each system of polynomial
equations system $P_{1}=\dotsb=P_{m}=0$ with Newton polyhedra $\Delta_{1},\dotsc,\Delta_{m}$
that contains $\leq n$ monomials, nonsingular at infinity in the
region $\arg z\in U$ of the space $(\mathbb{C}^{\ast})^{m}$, the
following relation holds:
\begin{equation}
\left|N(P,U)-S(P,U)\right|<\varphi(n,m)\Pi(\overline{\Delta}^{\ast},U).\label{eq:Khovanskii}
\end{equation}

\end{thm*}
Thus, taking a sequence of balls $U_{j}\subset\mathbb{T}^{m}$ shrinking
to $\{0\}\subset\mathbb{T}^{m}$, then $\Pi(\overline{\Delta}^{\ast},U_{j})\to1$
and $S(P,U_{j})\to0$. It follows that (\ref{eq:Khovanskii}) gives
\begin{equation}
\left|N_{\mathbb{R}}(P)\right|\leq\varphi(n,m),\label{eq:Khovanskii2}
\end{equation}
an upper bound for the number of real zeros of a polynomial system
$P$ which depends only on the number of nonzero terms $n$ and the
number of variables $m$.

While Khovanskii's result is concerned with the angular distribution
of zeros, the zero distribution of random complex fewnomials is concerned
with the radial distribution of zeros. For example, when $m=1$, the
limit of the normalized expected distribution of zeros of random $n$-nomials
is $\frac{1}{2\pi}\frac{\partial^{2}}{\partial\rho^{2}}F_{n}d\rho d\theta$.
But since $F_{n}$ only depends on the radial component $\rho$, we
consider the associated radial distribution with density function
$\frac{\partial^{2}}{\partial\rho^{2}}F_{n}$ with respect to the
measure $d\rho$ on $\mathbb{R}$. On a related note, we mention the
Fubini-Study metric $\omega_{\mathrm{FS}}=\frac{\sqrt{-1}}{2\pi}\partial\bar{\partial}\log\left(1+\left\Vert z\right\Vert ^{2}\right)$
has the associated radial distribution $[(1+e^{\rho})(1+e^{-\rho})]^{-1}d\rho$.

\begin{figure}
\includegraphics[width=0.85\textwidth]{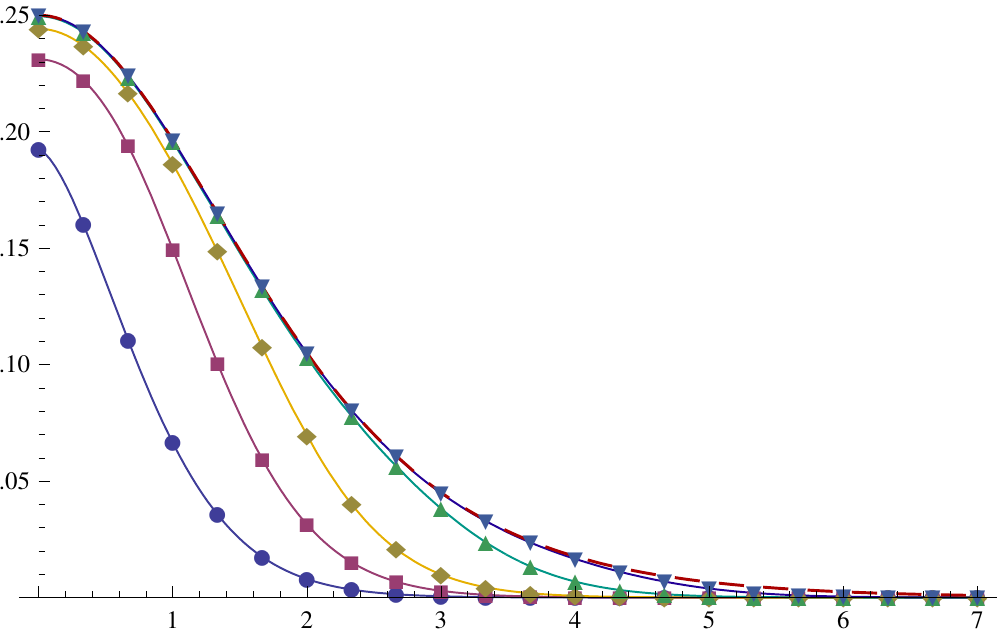}\caption{\label{fig:BeforeNormalization}Density functions $\frac{\partial^{2}}{\partial\rho^{2}}F_{n}$
for $n=2$ (\protect\includegraphics[bb=0bp 0bp 9bp 12bp]{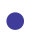}),
$n=4$ (\protect\includegraphics{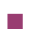}), $n=8$ (\protect\includegraphics{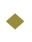}),
$n=32$ (\protect\includegraphics{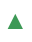}), and $n=128$
(\protect\includegraphics{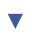}). Also the density of
$\omega_{\mathrm{FS}}$ (\protect\includegraphics[bb=0bp 18bp 23bp 27bp,clip]{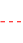}).}
\end{figure}

Plotting the density function $[(1+e^{\rho})(1+e^{-\rho})]^{-1}$
and the density function $\frac{\partial^{2}}{\partial\rho^{2}}F_{n}$
for various values of $n$ (Figure \ref{fig:BeforeNormalization})
suggests that perhaps $\frac{\partial^{2}}{\partial\rho^{2}}F_{n}$
is an increasing sequence of functions with the pointwise limit $[(1+e^{\rho})(1+e^{-\rho})]^{-1}$.

\begin{figure}
\includegraphics[width=0.85\textwidth]{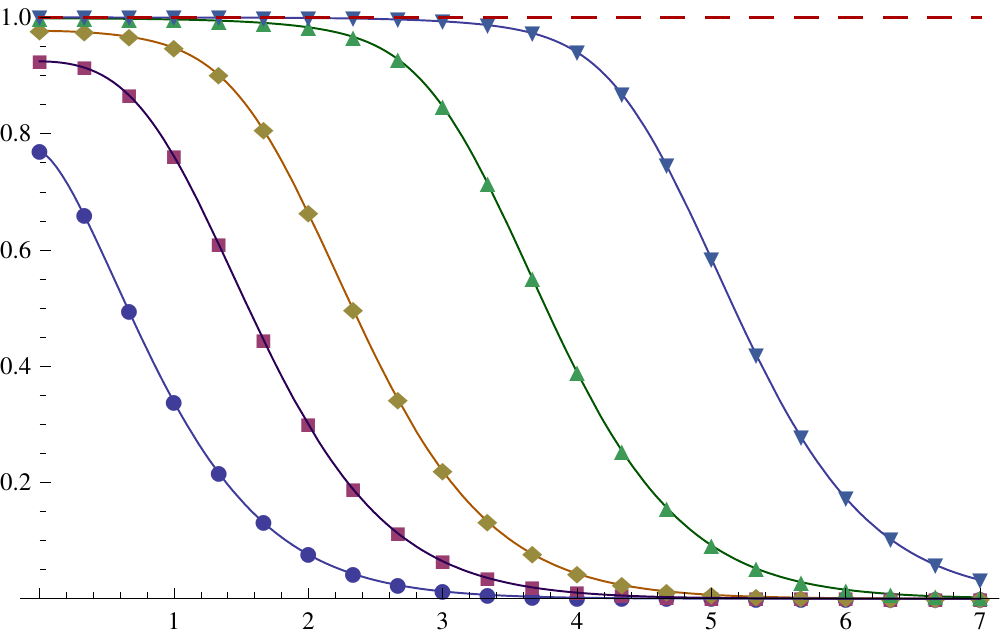}\caption{\label{fig:Quotient}Quotient of density functions $\frac{\partial^{2}}{\partial\rho^{2}}F_{n}$
over the density of $\omega_{\mathrm{FS}}$ for $n=2$ (\protect\includegraphics[bb=0bp 0bp 9bp 12bp]{DotBlueCircle}),
$n=4$ (\protect\includegraphics{DotPurpleSquare}), $n=8$ (\protect\includegraphics{DotYellowDiamond}),
$n=32$ (\protect\includegraphics{DotGreenTriangle}), $n=128$ (\protect\includegraphics{DotBlueTriangle})}
\end{figure}

Another interesting plot to consider is the quotient of $\frac{\partial^{2}}{\partial\rho^{2}}F_{n}$
over $[(1+e^{\rho})(1+e^{-\rho})]^{-1}$ (Figure \ref{fig:Quotient}).
The plot suggests that the quotient is a strictly decreasing function
for $\rho>0$ which remains close to $1$ for some initial interval
and then rapidly decay to $0$.

\begin{figure}
\includegraphics[width=0.85\textwidth]{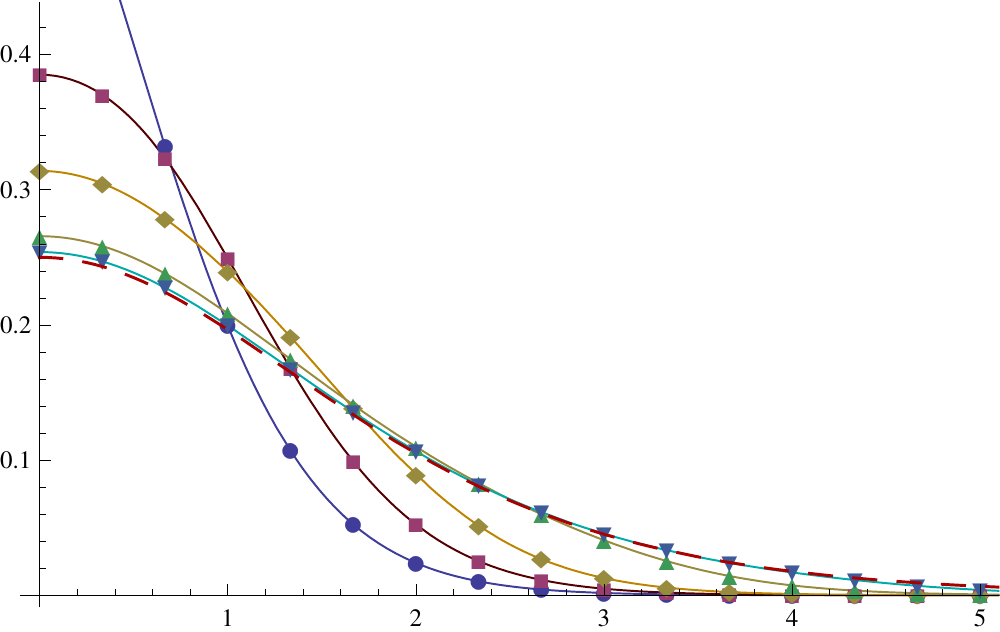}\caption{\label{fig:AfterNormalization}Normalized density functions $\frac{1}{\mathbb{E}_{n}(X)}\frac{\partial^{2}}{\partial\rho^{2}}F_{n}$
for $n=2$ (\protect\includegraphics[bb=0bp 0bp 9bp 12bp]{DotBlueCircle}),
$n=4$ (\protect\includegraphics{DotPurpleSquare}), $n=8$ (\protect\includegraphics{DotYellowDiamond}),
$n=32$ (\protect\includegraphics{DotGreenTriangle}), and $n=128$
(\protect\includegraphics{DotBlueTriangle}). Also the density of
$\omega_{\mathrm{FS}}$ (\protect\includegraphics[bb=0bp 18bp 23bp 27bp,clip]{DotDashed}).}
\end{figure}

Finally, we can consider the distribution with density $[\mathbb{E}_{n}(X)]^{-1}\frac{\partial}{\partial\rho^{2}}F_{n}$,
where $X$ is the integer-valued random variable on $\mathrm{Poly}(N)$
which maps a polynomial $p$ to the integer indicating the number
of roots of $p$ which are nonzero, $\mathbb{E}_{n}(X)=\lim_{N\to\infty}\frac{1}{N}\mathbb{E}_{N,n}(X)$,
and $\mathbb{E}_{N,n}(X)$ is the expected number of nonzero roots
of an $n$-nomial of degree $N$ where the spectra are chosen uniformly
at random. An exercise in combinatorics, we have 
\[
\mathbb{E}_{N,n}(X)=N\frac{n-1}{n+1}+2\frac{n}{n+1}.
\]
Note that we also have $\int_{\mathbb{C}}\mathbb{E}_{N,n}[Z_{P^{N}}]=\mathbb{E}_{N,n}(X)$.
Dividing both sides by $N$ and taking the limit as $N$ tends to
infinity, we obtain
\begin{alignat*}{1}
\int_{\mathbb{R}}\frac{\partial^{2}}{\partial\rho^{2}}F_{n}d\rho & =\mathbb{E}_{n}(X)=\frac{n-1}{n+1}.
\end{alignat*}
Thus the distribution with density $[\mathbb{E}_{n}(X)]^{-1}\frac{\partial}{\partial\rho^{2}}F_{n}$
is a probability distribution and comparing these distributions suggests
that the nonzero roots of an $n$-nomial are closer to the unit sphere,
when $n$ is small, and begin to spread out away from the unit sphere,
as $n$ increases (Figure \ref{fig:AfterNormalization}).

\section{\label{sec:Additional-Preliminaries}Rewriting the function $F_{n}$
(Part I)}

In Section 4.1 of ``Random complex fewnomials, I,'' Shiffman and
Zelditch obtain an alternate expression for $F_{n}$ and we duplicate
the work here for both the sake of clarity and the subsequent use
of the functions involved \cite{SZ}. Let $\Delta$ be the unit simplex
in $\mathbb{R}^{m}$. For $\rho\in\mathbb{R}^{m}$, we write $e^{\rho}=(e^{\rho_{1}},\dotsc,e^{\rho_{m}})$
and $\left|e^{\rho}\right|=\sum_{j=1}^{m}e^{\rho_{j}}$. Define
\begin{gather}
b:\Delta\times\mathbb{R}^{m}\to[0,\infty)\label{eq:b}\\
(\lambda,\rho)\mapsto\bigl\langle\hat{\lambda},\log\hat{\lambda}\bigr\rangle-\bigl\langle\rho,\lambda\bigr\rangle+\log(1+\left|e^{\rho}\right|)\nonumber 
\end{gather}
so that
\[
F_{n}(\rho)=\log(1+\left|e^{\rho}\right|)-\int_{\Delta^{n}}\min_{j=1,\dots,n}b(\lambda^{j},\rho)\, d\lambda^{1}\dotsm d\lambda^{n}.
\]

We proceed to use the following fact: Let $X$ be a nonnegative random
variable on a probability space $(\Omega,\mathcal{A},\mathcal{P}),$
and let $D_{X}(t)=\mathcal{P}(X\leq t)$ be its distribution function.
Then the expected value of $X$ is given by
\begin{equation}
\mathbb{E}(X)=\int Xd\mathcal{P}=\int_{0}^{\infty}[1-D_{x}(t)]dt.\label{eq:probabilityfact}
\end{equation}

If we let 
\[
D(t,\rho)=\mathcal{P}\{\lambda\in\Delta\mid b(\lambda,\rho)\leq t\}
\]
 be the distribution function for $b(\lambda)=b(\lambda,\rho)$, where
$d\mathcal{P}(\lambda)=d\lambda=m!d\lambda_{1}\cdots d\lambda_{m}$,
then the distribution function for the random variable
\begin{equation}
X(\lambda^{1},\dots,\lambda^{n})=\min\{b(\lambda^{1}),\dots,b(\lambda^{n})\}\label{eq:distributionfunction}
\end{equation}
on $\Delta^{n}$, with product measure $d\mathcal{P}(\lambda^{1})\dotsm d\mathcal{P}(\lambda^{n})$,
is given by
\begin{equation}
D_{X}=1-(1-D)^{n}.\label{eq:distributionofX}
\end{equation}

Using (\ref{eq:probabilityfact}) with (\ref{eq:distributionfunction})
and (\ref{eq:distributionofX}), we obtain 
\begin{equation}
F_{n}(\rho)-\log(1+\left|e^{\rho}\right|)=-\int_{\Delta^{n}}X\, d\lambda^{1}\dotsm d\lambda^{n}=-\int_{0}^{\infty}[1-D(t,\rho)]^{n}dt.\label{eq:main3}
\end{equation}

Having rewritten the potential function $F_{n}$ as an integral depending
on the distribution function $D$ of $b$, we are one step closer
to showing it is twice differentiably continuous. The function $D$,
however, is not useful to us in its current form and we are motivated
to determine another way to write it.

\section{\label{sec:Notation-and-Properties}Notation and properties of individual
functions}

In Section \ref{sec:Notation-and-Properties}, we introduce notation,
analyze the function $b$ which occurs in the definition of $D$,
discuss the case for $m=1$, and generalize the discussion to higher
dimension by introducing a function $h$ which behaves like an inverse
of $b$. We conclude the section by investigating the continuity of
various partial derivatives of two functions which we introduce in
Section \ref{sub:DefineFunctions} (namely the functions $B$ and
$h$).

\subsection{Notation and the function $b$}

Let $\Delta$ be the unit simplex in $\mathbb{R}^{m}$; explicitly,
we have 
\begin{alignat*}{1}
\Delta & =\{\lambda\in\mathbb{R}^{m}\mid\lambda_{i}\geq0\text{ for all }i=1,\dots,m\mbox{\text{ and }}\left|\lambda\right|\leq1\}\\
 & =\{\lambda\in\mathbb{R}^{m}\mid\lambda_{i}\geq0\text{ for all }i=0,1,\dots,m\},
\end{alignat*}
 where 
\[
\left|\lambda\right|\doteq\lambda_{1}+\dotsb+\lambda_{m}
\]
 and 
\begin{equation}
\lambda_{0}\doteq1-\left|\lambda\right|\Longleftrightarrow\sum_{i=0}^{m}\lambda_{i}=1.\label{eq:lambdazero}
\end{equation}

Then let $\Delta^{\circ}$ denote its interior and $\partial\Delta$
denote its boundary. For $i=0,1,\dots,m$, define the $i$-th facet
of $\Delta$ to be $\partial^{i}\Delta\doteq\{\lambda\in\Delta\mid\lambda_{i}=0\}$.
It follows that $\partial\Delta$ is a union of the facets $\partial^{i}\Delta$.
For $i\in\{0,1,\dots,m\}$, we also define $v_{i}\in\Delta$ to be
the vertex of $\Delta$ which is nonadjacent to the $i$-th facet.
By definition, the $i$-th vertex $v_{i}$ is the point $\lambda\in\Delta$
with $\lambda_{i}=1$ (equivalently, we may write $v_{i}=(\delta_{1i},\delta_{2i},\dots,\delta_{mi})$,
where $\delta_{ij}$ is the Kronecker delta function).

Now write $e^{\rho}=(e^{\rho_{1}},\dots,e^{\rho_{m}})$ and define
the (surjective) map
\begin{gather*}
\mu:\mathbb{R}^{m}\twoheadrightarrow\Delta^{\circ}\\
\rho\mapsto\frac{e^{\rho}}{1+\left|e^{\rho}\right|}.
\end{gather*}
Denoting the $i$-th component of $\mu$ by $\mu_{i}$, we have $\mu(\rho)=(\mu_{1}(\rho),\dots,\mu_{m}(\rho))$
and by (\ref{eq:lambdazero}) we obtain 
\[
\mu_{0}(\rho)=1-\left|\mu(\rho)\right|=\frac{1}{1+\left|e^{\rho}\right|}.
\]
Defining $\rho_{0}=0$, then we have 
\[
\mu_{i}(\rho)=\frac{e^{\rho_{i}}}{1+\left|e^{\rho}\right|}
\]
for $i=0,1,\dots,m$. We remark that the ability to generalize the
formula above to include the case for $i=0$ comes from associating
a polynomial $\sum_{\left|\alpha\right|\leq N}c_{\alpha}z^{\alpha}$
on $\mathbb{C}^{m}$ with the homogeneous polynomial $\sum_{\left|\alpha\right|\leq N}c_{\alpha}z_{0}^{\alpha_{0}}z^{\alpha}$
on $\mathbb{C}^{m+1}$. Then the original polynomial is obtained by
restricting the homogeneous polynomial to $z_{0}=1$. Since $\rho_{i}=\log\left|z_{i}\right|^{2}$,
we have $\rho_{0}=0$.

For $i=1,\dots,m$, we compute the first partial derivatives of $\mu_{i}=\mu_{i}(\rho)$
with respect to $\rho_{p}$. When $p=i$, we have
\[
\mu_{i,i}\doteq\frac{\partial}{\partial\rho_{i}}\mu_{i}=\frac{e^{\rho_{i}}}{1+\left|e^{\rho}\right|\vphantom{)^{2}}}-\frac{e^{\rho_{i}}e^{\rho_{i}}}{(1+\left|e^{\rho}\right|)^{2}}.
\]
When $p\neq i$, we have
\[
\mu_{i,p}\doteq\frac{\partial}{\partial\rho_{p}}\mu_{i}=-\frac{e^{\rho_{i}}e^{\rho_{p}}}{(1+\left|e^{\rho}\right|)^{2}}.
\]
And last, but not least, for $i=0$, we have
\[
\mu_{0,p}\doteq\frac{\partial}{\partial\rho_{p}}\mu_{0}=-\frac{e^{\rho_{0}}e^{\rho_{p}}}{(1+\left|e^{\rho}\right|)^{2}}.
\]
Conveniently, we summarize the above computations into one succinct
formula and write
\begin{equation}
\mu_{i,p}=\mu_{i\&p}-\mu_{i}\mu_{p},\label{eq:firstderivativeofmu}
\end{equation}
where 
\begin{alignat*}{1}
\mu_{i\&j} & \doteq\begin{cases}
\mu_{i} & \text{if }i=j,\\
0 & \text{otherwise}.
\end{cases}\\
 & =\delta_{ij}\mu_{i}
\end{alignat*}
As a side note, we observe that the indices exhibit a symmetry (e.g.,
$\mu_{i,p}=\mu_{p,i}$ for $i\neq0$) because $\mu_{i}=\frac{\partial}{\partial\rho_{i}}\log(1+\left|e^{\rho}\right|)$.

Next we analyze the function $b$ which we defined at (\ref{eq:b}). 
\begin{lem}
\label{lem:functionb}The function 
\begin{alignat*}{1}
b(\lambda,\rho) & =\bigl\langle\hat{\lambda},\log\hat{\lambda}\bigr\rangle-\bigl\langle\rho,\lambda\bigr\rangle+\log(1+\left|e^{\rho}\right|)\\
 & =\log(1+\left|e^{\rho}\right|)+\sum_{i=0}^{m}\left[\lambda_{i}\log\lambda_{i}-\rho_{i}\lambda_{i}\right],
\end{alignat*}
for $(\lambda,\rho)\in\Delta\times\mathbb{R}^{m}$, has the following
properties:

1) $b$ is smooth on $\Delta^{\circ}\times\mathbb{R}^{m}$,

2) $b$ is strictly convex with respect to the variable $\lambda$,

3) $b(\lambda,\rho)=0$ if and only $\lambda=\mu(\rho)$, and

4) $\max_{\lambda\in\Delta_{m}}b(\lambda,\rho)=\log(1+\left|e^{\rho}\right|)-\min\{\rho_{0},\rho_{1},\dots,\rho_{m}\}$. \end{lem}
\begin{proof}
The first property is clear.

To show that $b$ is strictly convex with respect to $\lambda$, we
compute the Hessian of $b(\cdot,\rho)$. Alternatively, one can also
check the definition of convexity and make use of the fact that $x\mapsto x\log x$
is a convex function for $x>0$. In any case, we move forward and
compute the first partial derivatives of $b$ in $\lambda$ to be
\[
b_{\lambda_{i}}(\lambda,\rho)=\log\lambda_{i}-\log\lambda_{0}-\rho_{i},\quad\text{for }(\lambda,\rho)\in\Delta^{\circ}\times\mathbb{R}^{m}
\]
and we compute the second partial derivatives of $b$ in $\lambda$
to be 
\[
b_{\lambda_{i}\lambda_{j}}(\lambda,\rho)=\frac{1}{\lambda_{i}}\delta_{ij}+\frac{1}{\lambda_{0}},\quad\text{for }(\lambda,\rho)\in\Delta^{\circ}\times\mathbb{R}^{m}.
\]
Thus, the Hessian of $b(\cdot,\rho)$ is given by
\[
H(\lambda)=\left[\begin{array}{cccc}
\lambda_{1}^{-1} & 0 & \cdots & 0\\
0 & \lambda_{2}^{-1} & \cdots & 0\\
\vdots & \vdots & \ddots & 0\\
0 & 0 & 0 & \lambda_{m}^{-1}
\end{array}\right]+\lambda_{0}^{-1}\left[\begin{array}{cccc}
\vphantom{\lambda_{1}^{-1}}1 & 1 & \cdots & 1\\
1 & \vphantom{\lambda_{1}^{-1}}1 & \cdots & 1\\
\vdots & \vdots & \ddots & \vdots\\
1 & 1 & \cdots & \vphantom{\lambda_{1}^{-1}}1
\end{array}\right].
\]
Then for any $v\in\mathbb{R}^{m}\setminus\{0\}$ and $\lambda\in\Delta^{\circ}$,
we have
\begin{alignat*}{1}
v^{\intercal}H(\lambda)\, v & =\sum_{i=1}^{m}\lambda_{i}^{-1}v_{i}^{2}+\lambda_{0}^{-1}\left[\sum_{i=1}^{n}v_{i}\right]^{2}>0
\end{alignat*}
and conclude the Hessian of $b(\cdot,\rho)$ is positive definite
on the interior of the unit simplex. Hence, $b(\cdot,\rho)$ is strictly
convex on the unit simplex.

For the third property, it is an immediate check that $b\bigl(\mu(\rho),\rho\bigr)=0$
and we are left to show the forward implication. We fix $\rho\in\mathbb{R}^{m}$
and solve the system of equations $b_{\lambda_{i}}(\lambda,\rho)=0$
in order to determine critical points (for maxima and minima) of $b(\cdot,\rho)$.
For $i=1,\dotsc,m$, we immediately obtain $\lambda_{i}=\lambda_{0}e^{\rho_{i}}$.
Together with $\sum_{i=0}^{m}\lambda_{i}=1$, we have $\lambda_{0}+\sum_{i=1}^{m}(\lambda_{0}e^{\rho_{i}})=1$
and we obtain $\lambda_{0}=\mu_{0}(\rho)$. Thus, we have $\lambda_{i}=\mu_{0}(\rho)e^{\rho_{i}}=\mu_{i}(\rho)$
for $i=1,\dotsc,m$ and we conclude $b(\cdot,\rho)$ has a critical
point at $\mu(\rho)$. We deduce from the strict convexity of $b(\cdot,\rho)$
that $b(\cdot,\rho)$ has a unique global minimum at $\mu(\rho)$.
It follows that $b(\lambda,\rho)=0$ only if $\lambda=\mu(\rho)$,
as desired.

Finally, because $b$ is convex on $\Delta$, it must take its maximum
at one of the vertices $v_{i}$ of $\Delta$. For $i=0,1,\dots m$,
we have $b(v_{i},\rho)=\log(1+\left|e^{\rho}\right|)-\rho_{i}$. Therefore,
we take the maximum over $i=0,1,\dots,m$ and obtain the fourth and
final property.
\end{proof}

\subsection{\label{sub:Motivation}Motivation}

While the cases for $m=1$ and $m=2$ do not require special attention,
our ability to visualize real-valued functions on $\mathbb{R}$ and
$\mathbb{R}^{2}$ come in handy to understand the work done for arbitrary
dimensions. For $m=1$, Lemma \ref{lem:functionb} implies the function
\[
b(\lambda,\rho)=\lambda\log\lambda+(1-\lambda)\log(1-\lambda)-\lambda\rho+\log(1+e^{\rho})
\]
is convex in $\lambda$ with an absolute minimum at $\lambda=\mu(\rho)=e^{\rho}[1+e^{\rho}]^{-1}$.
Consequently, the function $b(\cdot,\rho)$ has an inverse branch
$g_{-1}(\cdot,\rho)$ on the interval $[0,\mu(\rho)]$ and an inverse
branch $g_{1}(\cdot,\rho)$ on the interval $[\mu(\rho),1]$. Taking
these branches and extending their domain continuously to $[0,\infty)$,
we obtain 
\[
g_{-1}(t,\rho)=\begin{cases}
\lambda, & \text{if }t\in[0,b(0,\rho)],\, b(\lambda,\rho)=t,\text{ and }\lambda\in[0,\mu(\rho)]\\
0, & \text{if }t\in[b(0,\rho),\infty)
\end{cases}
\]
and 

\[
g_{1}(t,\rho)=\begin{cases}
\lambda, & \text{if }t\in[0,b(1,\rho)],\, b(\lambda,\rho)=t,\text{ and }\lambda\in[\mu(\rho),1]\\
1, & \text{if }t\in[b(1,\rho),\infty),
\end{cases}
\]
so that
\begin{align*}
D(t,\rho) & =\mathcal{P}\{\lambda\in\Delta\mid b(\lambda,\rho)\leq t\}\\
 & =g_{1}(t,\rho)-g_{-1}(t,\rho).
\end{align*}
This evaluation of $D$ as a difference of inverse branches is due
to Shiffman and Zelditch and found in Section 4.1.1 of their paper
``Random complex fewnomials, I,'' \cite{SZ}. When looking to generalize
the idea to higher dimensions, we further write
\begin{align*}
D(t,\rho) & =g_{1}(t,\rho)-\mu(\rho)+\mu(\rho)-g_{-1}(t,\rho)\\
 & =\left|g_{1}(t,\rho)-\mu(\rho)\right|+\left|g_{-1}(t,\rho)-\mu(\rho)\right|.
\end{align*}

With the above in mind, we consider the case for $m=2$. As with the
case for $m=1$, Lemma \ref{lem:functionb} implies the function $b(\cdot,\rho)$
is convex with an absolute minimum at $\lambda=\mu(\rho)$. Consequently,
given any unit direction $u\in\mathbb{R}^{2}$, the function $b(\cdot,\rho)$
has an inverse branch $g_{u}(\cdot,\rho)$ with respect to the direction
$u$. Taking these branches and extending their domain continuously
to $[0,\infty)$, we obtain
\[
g_{u}(t,\rho)=\begin{cases}
\lambda, & \text{if }t\in[0,b(x_{u},\rho)],\, b(\lambda,\rho)=t,\text{ and }\lambda=\mu+su\text{ for some }s\geq0\\
x_{u} & \text{if }t\in[b(x_{u},\rho),\infty),
\end{cases}
\]
where $x_{u}$ is the unique element in the boundary of the unit simplex
$\partial\Delta$ such that $x_{u}=\mu+s_{u}u$ for some $s_{u}>0$.
(We note that such an $x_{u}$ exists and is unique because the unit
simplex $\Delta$ is convex and $\mu\in\Delta$.) Then, loosely speaking,
the function $D$ in the case for $m=2$ would be the integral of
the lengths $\left\Vert f_{u}(t,\rho)-\mu(\rho)\right\Vert $ over
all unit directions $u$ in $\mathbb{R}^{2}$. However, the definition
of $f_{u}$ is quite cumbersome and we are actually interested in
the lengths $\left\Vert f_{u}(t,\rho)-\mu(\rho)\right\Vert $. Thus,
an small improvement is to consider the functions 
\begin{equation}
\tilde{f}_{u}(t,\rho)=\begin{cases}
s, & \text{if }t\in[0,b(x_{u},\rho)]\text{ and }b(\mu+su,\rho)=t\text{ for some }s\geq0\\
s_{u}, & \text{if }t\in[b(x_{u},\rho),\infty),
\end{cases}\label{eq:auxiliaryfunction}
\end{equation}
where $x_{u}$ is the unique element in the boundary of the unit simplex
$\partial\Delta$ such that $x_{u}=\mu+s_{u}u$ for some $s_{u}>0$.
This will be our initial approach in Section \ref{sub:DefineFunctions}
as it allows us to apply the implicit function theorem. For $m=1$,
we have
\[
D(t,\rho)=\tilde{f}_{1}(t,\rho)+\tilde{f}_{-1}(t,\rho).
\]

However, when we want to actually take derivatives and set up an integral,
it is convenient to parametrize using the boundary points $x$ in
$\partial\Delta$, instead of using unit directions $u$ in $\mathbb{R}^{2}$.
Explicitly, we consider the functions
\begin{equation}
f_{x}(t,\rho)=\begin{cases}
s, & \text{if }t\in[0,b(x,\rho)]\text{ and }b((1-s)\mu+sx,\rho)=t\\
1, & \text{if }t\in[b(x,\rho),\infty).
\end{cases}\label{eq:mainfunction}
\end{equation}
Note, however, the formulation of this family of functions $\{f_{x}\}_{x\in\partial\Delta}$
is simpler than the formulation of the family of functions $\{\tilde{f}_{u}\}_{u\in S^{m}}$.
In addition, the range of $f_{x}(\cdot,\rho)$, for any $x\in\partial\Delta$,
is the interval $[0,1]$, while the range of $\tilde{f}_{u}(\cdot,\rho)$
varies as $u$ varies. Noting that $(1-s)\mu+sx=\mu+s(x-\mu)$, we
obtain 
\begin{equation}
f_{x}(t,\rho)=\frac{1}{\left\Vert x-\mu\right\Vert }\tilde{f}_{u}(t,\rho),\quad\text{where }u=\frac{x-\mu}{\left\Vert x-\mu\right\Vert },\label{eq:relationship-1}
\end{equation}
or equivalently
\[
\tilde{f}_{u}(t,\rho)=\left\Vert x_{u}-\mu\right\Vert f_{x_{u}}(t,\rho),\quad\text{where }x_{u}=\mu+su\in\partial\Delta,\, s>0,
\]
which gives a one-to-one correspondence between the family of functions
$\{\tilde{f}_{u}(\cdot,\rho)\}_{u\in S^{1}}$ and $\{f_{x}(\cdot,\rho)\}_{x\in\partial\Delta}$.
For $m=1$, we have
\begin{alignat}{1}
D(t,\rho) & =\tilde{f}_{1}(t,\rho)+\tilde{f}_{-1}(t,\rho)\label{eq:Dwrtf}\\
 & =\left|1-\mu_{1}\right|f_{1}(t,\rho)+\left|0-\mu_{1}\right|f_{0}(t,\rho)\nonumber \\
 & =\mu_{0}\cdot f_{1}(t,\rho)+\mu_{1}\cdot f_{0}(t,\rho).\nonumber 
\end{alignat}

This completes our motivation for the functions we plan to define.

\subsection{\label{sub:DefineFunctions}Defining the Functions $\Lambda$, $B$,
and $h$}

First, we define some auxiliary functions, which after this section,
we no longer consider. Specifically, we will set out to rigorously
define a function $\tilde{h}$, in fact $\tilde{h}(t,\rho,v)=\tilde{f}_{v/\left\Vert v\right\Vert }(t,\rho)$
(\ref{eq:auxiliaryfunction}), only to rigorously define a function
$h$, in fact $h(t,\rho,x)=f_{x}(t,\rho)$ (\ref{eq:mainfunction}),
which inherits desired smoothness properties from $h$ by (\ref{eq:relationship-1}).

We start by defining the (surjective) map 
\begin{gather*}
\tilde{\Lambda}:\mathbb{R}\times\mathbb{R}^{m}\times(\mathbb{R}^{m}\setminus\{0\})\twoheadrightarrow\mathbb{R}^{m}\\
(\tilde{s},\rho,v)\mapsto\mu(\rho)+\tilde{s}\frac{v}{\bigl\Vert v\bigr\Vert}
\end{gather*}
and consider the function 
\[
\tilde{B}(\tilde{s},\rho,v)=b\bigl(\tilde{\Lambda}(\tilde{s},\rho,v),\rho\bigr),
\]
whose domain we denote by $\tilde{S}$. By Lemma \ref{lem:functionb},
the function $\tilde{B}$ is smooth on the interior of $\tilde{S}$
and $\tilde{B}(\cdot,\rho,v)$ is a strictly increasing function for
$\tilde{s}>0$. Combining these facts with the implicit function theorem,
there exists an open set $\tilde{H}$ in $(0,\infty)\times\mathbb{R}^{m}\times(\mathbb{R}^{m}\setminus\{0\})$
and a smooth function $\tilde{h}:\tilde{H}\to[0,\infty)$ such that
\begin{equation}
\tilde{B}\bigl(\tilde{h}(t,\rho,v),\rho,v\bigr)=t,\quad\text{for }(t,\rho,v)\in\tilde{H}.\label{eq:temporaryequality}
\end{equation}
For the sake of completeness, we give a description of the set $\tilde{H}$,
but note that it is not essential, but note that it has little importance.
For $\rho\in\mathbb{R}^{m}$, $v\in\mathbb{R}^{m}\setminus\{0\}$,
and $(t,\rho,v)\in\tilde{H}$, then $t$ must belong to the interior
of the range of $\tilde{B}(\cdot,\rho,v)\vert_{\tilde{s}\geq0}$.
In other words, we have 
\[
\tilde{H}=\{(t,\rho,v)\mid\tilde{B}(\tilde{s},\rho,v)=t\mbox{\text{ for some }}(\tilde{s},\rho,v)\in\tilde{S}\text{ with }\tilde{s}>0\}.
\]

Now fix $x\in\partial\Delta$, take (\ref{eq:temporaryequality}),
and set $v=x-\mu(\rho)\in\mathbb{R}^{m}\setminus\{0\}$. Then we have
\[
b\left(\mu+\tilde{h}(t,\rho,x-\mu)\frac{x-\mu}{\left\Vert x-\mu\right\Vert },\rho\right)=t,
\]
or
\[
b\bigl(\mu+h(t,\rho,x)(x-\mu),\rho\bigr)=t,\quad\text{where }h(t,\rho,x)=\frac{1}{\bigl\Vert x-\mu\bigr\Vert}\tilde{h}(t,\rho,x-\mu).
\]

Similar to replacing $\{f_{u}\}_{u\in S^{1}}$ with $\{f_{x}\}_{x\in\Delta}$
in Section \ref{sub:Motivation}, we will replace $\tilde{\Lambda}$,
$\tilde{B}$, $\tilde{H}$, and $\tilde{h}$, with their analogs $\Lambda$,
$B$, $H$, and $h$. We define the (surjective) map
\begin{gather*}
\Lambda:[0,1]\times\mathbb{R}^{m}\times\partial\Delta\twoheadrightarrow\Delta\\
(s,\rho,x)\mapsto\mu(\rho)+s\left(x-\mu(\rho)\right)\\
\hphantom{(s,\rho,x)}=(1-s)\mu(\rho)+sx\hphantom{\mu},
\end{gather*}
the function
\begin{gather*}
B:[0,1]\times\mathbb{R}^{m}\times\partial\Delta\to[0,\infty)\\
(s,\rho,x)\mapsto b\bigl(\Lambda(s,\rho,x),\rho\bigr),
\end{gather*}
the domain 
\[
H=\{(t,\rho,x)\in(0,\infty)\times\mathbb{R}^{m}\times\partial\Delta\mid0<t<b(x,\rho)\},
\]
and the function
\begin{gather*}
h:H\twoheadrightarrow(0,1)\\
(t,\rho,x)\mapsto s\text{ such that }B(s,\rho,x)=t.
\end{gather*}

Using Lemma \ref{lem:functionb} and the observation that the mixed
partial derivatives of the map $\Lambda$ in $s$ and $\rho$ of all
orders are continuous on $[0,1]\times\mathbb{R}^{m}\times\partial\Delta$,
we find the mixed partial derivatives of $B$ in $s$ and $\rho$
(of all orders) are continuous on $[0,1)\times\mathbb{R}^{m}\times\partial\Delta$.
Lemma \ref{lem:functionb} also implies that $B$ is a strictly increasing
function with respect to $s$. 

Similarly for $h$, we note that the partial derivatives of the maps
$(\rho,x)\mapsto x-\mu(\rho)$ and $(\rho,x)\mapsto\left\Vert x-\mu(\rho)\right\Vert ^{-1}$
in $\rho$ of all orders are continuous on $\mathbb{R}^{m}\times\partial\Delta$.
Then the mixed partial derivatives of $h$ in $t$ and $\rho$ (of
all orders) are continuous on $H$, since $\tilde{h}\in\mathcal{C}^{\infty}(\tilde{H})$.
We also have the function $h$ is strictly increasing with respect
to $t$, because the function $B$ is strictly increasing with respect
to $s$.

\subsection{Properties of the function $B$}

As the definition of $h$ is implicitly defined by $B$, we proceed
by first investigating $B$ and its derivatives along the boundary
of its domain. A summary may be found in Section \vref{sub:SummarizeB}.

\subsubsection{\label{sub:AlongSis0}Along $s=0$.}

As mentioned in Section \ref{sub:DefineFunctions}, the mixed partial
derivatives of $B$ in $s$ and $\rho$ (of all orders) are continuous
on $[0,1)\times\mathbb{R}^{m}\times\partial\Delta$. Thus, along $s=0$,
we simply look to evaluate the derivatives of $B$.

The first useful fact in our computations follows from the definition
of $\mu_{0}$ and $\mu_{i}$. For $i=0,1,\dots m$, we have 
\begin{equation}
\log[\mu_{0}]=\log[\mu_{i}]-\rho_{i}.\label{eq:fact1}
\end{equation}
Taking the first partial derivative of $B$ with respect to $s$,
we obtain
\[
B_{s}(s,\rho,x)=\sum_{i=0}^{m}\Bigl[\log[(1-s)\mu_{i}+sx_{i}]-\rho_{i}\Bigr](x_{i}-\mu_{i})
\]
and note that it is nonnegative, equaling $0$ if and only if $s=0$,
because $B$ is strictly increasing for $s>0$. Checking $B_{s}(0,\rho,x)=0$,
we write
\begin{alignat}{1}
B_{s}(0,\rho,x) & =\sum_{i=0}^{m}\Bigl[\log[\mu_{i}]-\rho_{i}\Bigr](x_{i}-\mu_{i})\label{eq:lim0Bs}\\
 & \overset{\ref{eq:fact1}}{=}\log[\mu_{0}]\cdot\sum_{i=0}^{m}(x_{i}-\mu_{i})=0,\nonumber 
\end{alignat}
where the last equality is given by 
\begin{equation}
\sum_{i=0}^{m}(x_{i}-\mu_{i})=\sum_{i=0}^{m}x_{i}-\sum_{i=0}^{m}\mu_{i}\overset{\ref{eq:lambdazero}}{=}1-1=0.\label{eq:fact2}
\end{equation}

While evaluating $B_{s}$ at $s=0$ took several steps, no heavy computation
is required for evaluating $B_{ss}$ along $s=0$. We have
\[
B_{ss}(s,\rho,x)=\sum_{i=0}^{m}\frac{(x_{i}-\mu_{i})^{2}}{(1-s)\mu_{i}+sx_{i}}>0,
\]
and so 
\begin{equation}
B_{ss}(0,\rho,x)=\sum_{i=0}^{m}\frac{(x_{i}-\mu_{i})^{2}}{\mu_{i}}.\label{eq:lim0Bss}
\end{equation}

Now let us differentiate $B$ with respect to $\rho_{p}$. As we apply
the chain rule, we note that
\[
b_{\rho_{i}}(\lambda,\rho)=-\lambda_{i}+\mu_{i}(\rho),\quad\text{for }(\lambda,\rho)\in\Delta\times\mathbb{R}^{m}
\]
and so
\[
B_{\rho_{p}}(s,\rho,x)=-s(x_{p}-\mu_{p})+(1-s)\sum_{i=0}^{m}\Bigl[\log[(1-s)\mu_{i}+sx_{i}]-\rho_{i}\Bigr]\mu_{i,p}.
\]
Evaluating the derivative for $s=0$, we have 
\begin{alignat*}{1}
B_{\rho_{p}}(0,\rho,x) & =\sum_{i=0}^{m}\Bigl[\log[\mu_{i}]-\rho_{i}\Bigr]\mu_{i,p}\\
 & \overset{\ref{eq:fact1}}{=}\log[\mu_{0}]\cdot\sum_{i=0}^{m}\mu_{i,p}=0,
\end{alignat*}
where the last equality is given by
\begin{equation}
\sum_{i=0}^{m}\mu_{i,p}=\frac{\partial}{\partial\rho_{p}}\left[\sum_{i=0}^{n}\mu_{i}\right]\overset{\ref{eq:lambdazero}}{=}0.\label{eq:fact3}
\end{equation}

Next, we consider the partial derivative
\begin{multline*}
B_{\rho_{p}s}(s,\rho,x)=-(x_{p}-\mu_{p})+\sum_{i=0}^{m}\left[-\log[(1-s)\mu_{i}+sx_{i}]+\rho_{i}+\frac{(1-s)(x_{i}-\mu_{i})}{(1-s)\mu_{i}+sx_{i}}\right]\mu_{i,p}
\end{multline*}
so that
\begin{alignat}{1}
B_{\rho_{p}s}(0,\rho,x) & =-(x_{p}-\mu_{p})+\sum_{i=0}^{m}\left[-\log[\mu_{i}]+\rho_{i}+\frac{(x_{i}-\mu_{i})}{\mu_{i}}\right]\mu_{i,p}\label{eq:lim0Brhops}\\
 & \overset{\ref{eq:fact1}}{=}-(x_{p}-\mu_{p})+\sum_{i=0}^{m}\left[-\log[\mu_{0}]+\frac{(x_{i}-\mu_{i})}{\mu_{i}}\right]\mu_{i,p}\nonumber \\
 & \overset{\ref{eq:fact3}}{=}-(x_{p}-\mu_{p})+\sum_{i=0}^{m}\left[\frac{(x_{i}-\mu_{i})}{\mu_{i}}\right]\mu_{i,p}=0,\nonumber 
\end{alignat}
where the last equality is given by
\[
\sum_{i=0}^{m}\frac{x_{i}-\mu_{i}}{\mu_{i}}\mu_{i,p}\overset{\ref{eq:firstderivativeofmu}}{=}\sum_{i=0}^{m}(x_{i}-\mu_{i})(\delta_{ip}-\mu_{p})\overset{\ref{eq:fact2}}{=}x_{p}-\mu_{p}.
\]

Finally, we have the partial derivative
\begin{multline*}
B_{\rho_{p}\rho_{q}}(s,\rho,x)=(2s-1)\mu_{p,q}+\sum_{i=0}^{m}\frac{\mu_{i,q}\mu_{i,p}\cdot(1-s)^{2}}{(1-s)\mu_{i}+sx_{i}}\\
+(1-s)\sum_{i=0}^{m}\Bigl[\log[(1-s)\mu_{i}+sx_{i}]-\rho_{i}\Bigr]\mu_{i,p,q}
\end{multline*}
 for which 
\begin{alignat*}{1}
B_{\rho_{p}\rho_{q}}(0,\rho,x) & =-\mu_{p,q}+\sum_{i=0}^{m}\frac{\mu_{i,q}\mu_{i,p}}{\mu_{i}}+\sum_{i=0}^{m}\Bigl[\log[\mu_{i}]-\rho_{i}\Bigr]\mu_{i,p,q}.\\
 & \overset{\ref{eq:fact1}}{=}-\mu_{p,q}+\sum_{i=0}^{m}\frac{\mu_{i,q}\mu_{i,p}}{\mu_{i}}+\log[\mu_{0}]\cdot\sum_{i=0}^{m}\mu_{i,p,q}\\
 & =-\mu_{p,q}+\sum_{i=0}^{m}\mu_{i,q}\cdot\frac{\mu_{i,p}}{\mu_{i}}=0,
\end{alignat*}
where the third equality is given by
\[
\sum_{i=0}^{m}\mu_{i,p,q}=\frac{\partial^{2}}{\partial\rho_{p}\partial\rho_{q}}\left[\sum_{i=0}^{m}\mu_{i}\right]\overset{\ref{eq:lambdazero}}{=}0,
\]
and the last equality is given by 
\[
\sum_{i=0}^{m}\mu_{i,q}\cdot\frac{\mu_{i,p}}{\mu_{i}}\overset{\ref{eq:firstderivativeofmu}}{=}\sum_{i=0}^{m}\mu_{i,q}\cdot(\delta_{ip}-\mu_{p})\overset{\ref{eq:fact3}}{=}\mu_{p,q}.
\]

\subsubsection{Along $s=1$.\label{sub:AlongSis1}}

In contrast to Section \ref{sub:AlongSis0} where we simply evaluate
various partial derivatives of $B$ along $s=0$, this section deals
with the limiting behavior of the partial derivatives as $s$ tends
to $1$.

Then let us begin by fixing $\tilde{\rho}\in\mathbb{R}^{m}$ and $\tilde{x}\in\partial\Delta$.
We also consider the variables $s\in(0,1)$, $\rho\in\mathbb{R}^{m}$,
and $x\in\partial\Delta$.

For the limit of $B_{s}$ as $(s,\rho,x)$ tends to $(1,\tilde{\rho},\tilde{x})$,
we note that the summands $(x_{i}-\mu_{i})\log[(1-s)\mu_{i}+sx_{i}]$
of $B_{s}$, for all $i=0,1,\dots,m$ such that $\tilde{x}_{i}\neq0$,
and the summands $-\rho_{i}(x_{i}-\mu_{i})$ of $B_{s}$, for all
$i=0,1,\dots,m$, have finite limits as $(s,\rho,x)$ tends to $(1,\tilde{\rho},\tilde{x})$.
We next consider the subset $S\subset\{0,1,\dotsc,m\}$ of indices
for which $\tilde{x}_{i}=0$. Then there exists $\delta_{1}>0$ such
that $x_{i}-m_{i}<0$ for all $i\in S$ and $\left\Vert (s,\rho,x)-(1,\tilde{\rho},\tilde{x})\right\Vert <\delta_{1}$.
Thus all the summands of $\sum_{i\in S}(x_{i}-\mu_{i})\log[(1-s)\mu_{i}+sx_{i}]$
are positive and we have 
\begin{alignat*}{1}
\sum_{i\in S}(x_{i}-\mu_{i})\log[(1-s)\mu_{i}+sx_{i}] & \geq(x_{k}-\mu_{k})\log[(1-s)\mu_{k}+sx_{k}]\\
 & =-\mu_{k}\cdot\log[(1-s)\mu_{k}],
\end{alignat*}
where $k$ is an index for which $x_{k}=0$. For any $M>0$, there
exists $\delta_{2}$ such that $-\mu_{k}\cdot\log[(1-s)\mu_{k}]>M$
for all $k\in S$ and $\left\Vert (s,\rho,x)-(1,\tilde{\rho},\tilde{x})\right\Vert <\delta_{2}$.
We conclude that 
\begin{equation}
\lim_{(s,\rho,x)\to(1,\tilde{\rho},\tilde{x})}B_{s}(s,\rho,x)=\infty.\label{eq:lim1Bs}
\end{equation}

For the limit of $B_{ss}$ as $(s,\rho,x)$ tends to $(1,\tilde{\rho},\tilde{x})$,
we first note that all the summands are positive and proceed similar
to the proof of (\ref{eq:lim1Bs}). We have
\begin{alignat*}{1}
B_{ss}(s,\rho,y) & \geq(x_{k}-\mu_{k})^{2}[(1-s)\mu_{k}+sx_{k}]^{-1}\\
 & =(-\mu_{k})^{2}[(1-s)\mu_{k}]^{-1}=(1-s)^{-1}\mu_{k},
\end{alignat*}
where $k$ is an index for which $x_{k}=0$. For any $M>0$, there
exists a $\delta>0$ such that $\mu_{k}\cdot(1-s)^{-1}>M$ for all
$k\in\{0,1,\dots,m\}$ and $\left\Vert (s,\rho,x)-(1,\rho^{0},x^{0})\right\Vert <\delta$.
Thus, we have 
\[
\lim_{(s,\rho,x)\to(1,\tilde{\rho},\tilde{x})}B_{ss}(s,\rho,x)=\infty.
\]

Next we look at the limit of $B_{\rho_{p}}$ as $(s,\rho,x)$ tends
to $(1,\tilde{\rho},\tilde{x})$. We note that the summation $\sum_{i=0}^{m}(1-s)(-\rho_{i})\mu_{i,\rho}$
tends to $0$ and the term $-s(x_{p}-\mu_{p})$ tends to $-[\tilde{x}_{p}-\mu_{p}(\tilde{\rho})]$
as $(s,\rho,x)$ tends to $(1,\tilde{\rho},\tilde{x})$. Thus it remains
to investigate the limit of $\sum_{i=0}^{m}(1-s)\mu_{i,p}\cdot\log[(1-s)\mu_{i}+sx_{i}]$
as $(s,\rho,x)$ tends to $(1,\tilde{\rho},\tilde{x})$. We have
\begin{alignat*}{1}
\left|(1-s)\sum_{i=0}^{m}\mu_{i,p}\cdot\log[(1-s)\mu_{i}+sx_{i}]\right| & \leq(1-s)\sum_{i=0}^{m}\left|\mu_{i,p}\cdot\log[(1-s)\mu_{i}+sx_{i}]\right|\\
 & =-(1-s)\sum_{i=0}^{m}\left|\mu_{i,p}\right|\log[(1-s)\mu_{i}+sx_{i}]\\
 & \leq-(1-s)\sum_{i=0}^{m}\left|\mu_{i,p}\right|\log[(1-s)\mu_{i}],
\end{alignat*}
where the last inequality is due to the inequality $(1-s)\mu_{i}+sx_{i}\geq(1-s)\mu_{i}$
and the fact that $\log(\cdot)$ is an increasing function. Next,
we note that there exists a neighborhood $N$ of $(1,\tilde{\rho},\tilde{x})$
such that for all $(s,\rho,x)$ in $N$ and all $i=0,1,\dotsc,m$
we have $\left|\mu_{i,p}\right|<M$ and $-\log[\mu_{i}]<M$ for some
$M>0$. Letting $\varepsilon>0$, there also exists a neighborhood
$\tilde{N}$ of $(1,\tilde{\rho},\tilde{x})$ such that $-(1-s)\log[(1-s)]<\varepsilon$
and $(1-s)<\varepsilon$ for all $(s,\rho,x)$ in $\tilde{N}$. It
follows that the summation $\sum_{i=0}^{m}(1-s)\mu_{i,p}\cdot\log[(1-s)\mu_{i}+sx_{i}]$
tends to $0$ as $(s,\rho,x)$ tends to $(1,\tilde{\rho},\tilde{x})$.
Therefore, we obtain 
\begin{equation}
\lim_{(s,\rho,x)\to(1,\tilde{\rho},\tilde{x})}B_{\rho_{p}}(s,\rho,x)=-[\tilde{x}_{p}-\mu_{p}(\tilde{\rho})].\label{eq:lim1Brhop}
\end{equation}

Finally, for the limit of $B_{\rho_{p}\rho_{q}}$, we note that the
term $(2s-1)\mu_{p,q}$ tends to $\mu_{p,q}(\tilde{\rho})$ as $(s,\rho,x)$
tends to $(1,\tilde{\rho},\tilde{x})$ and the summation $(1-s)\sum_{i=0}^{m}\Bigl[\log[(1-s)\mu_{i}+sx_{i}]-\rho_{i}\Bigr]\mu_{i,p,q}$
tends to $0$ for the same reason that $\sum_{i=0}^{m}(1-s)\Bigl[\log[(1-s)\mu_{i}+sx_{i}]-\rho_{i}\Bigr]m_{i,p}$
tends to $0$ in the proof of (\ref{eq:lim1Brhop}). We are left to
investigate the limit of the remaining summand $\sum_{i=0}^{m}\mu_{i,q}\mu_{i,p}\cdot(1-s)^{2}[(1-s)\mu_{i}+sx_{i}]^{-1}$
as $(s,\rho,x)$ tends to $(1,\tilde{\rho},\tilde{x})$. First taking
the absolute value and applying the triangle inequality, we have 
\begin{alignat*}{1}
\left|\sum_{i=0}^{m}\frac{\mu_{i,q}\mu_{i,p}\cdot(1-s)^{2}}{(1-s)\mu_{i}+sx_{i}}\right| & \leq\sum_{i=0}^{m}\frac{\left|\mu_{i,q}\mu_{i,p}\right|(1-s)^{2}}{(1-s)\mu_{i}+sx_{i}}\\
 & \leq\sum_{i=0}^{m}\frac{\left|\mu_{i,q}\mu_{i,p}\right|(1-s)^{2}}{(1-s)\mu_{i}}\\
 & =(1-s)\sum_{i=0}^{m}\frac{\left|\mu_{i,q}\mu_{i,p}\right|}{\mu_{i}}.
\end{alignat*}
Thus, the summand tends to $0$ as $(s,\rho,x)$ tends to $(1,\tilde{\rho},\tilde{x})$,
because $\sum_{i=0}^{m}\left|\mu_{i,q}\mu_{i,p}\right|[\mu_{i}]^{-1}$
is uniformly bounded in a small enough neighborhood of $(1,\tilde{\rho},\tilde{x})$.
We conclude that 
\[
\lim_{(s,\rho,x)\to(1,\tilde{\rho},\tilde{x})}B_{\rho_{p}\rho_{q}}(s,\rho,x)=\mu_{p,q}(\tilde{\rho}).
\]

\subsubsection{Summary for $B$\label{sub:SummarizeB}}

Here we summarize the computations for $B$ and its partial derivatives.
Throughout this summary, we write $f(1,\tilde{\rho},\tilde{x})$ to
mean ${\displaystyle \lim_{(s,\rho,x)\to(1,\tilde{\rho},\tilde{x})}}f(s,\rho,x)$.

The function 
\[
B(s,\rho,x)=b(\Lambda(s,\rho,x),\rho)
\]
 is continuous on $[0,1]\times\mathbb{R}^{m}\times\partial\Delta$
and we have
\[
B(0,\rho,x)=0,\,\text{and }B(1,\rho,x)=b(x,\rho).
\]

Taking a partial derivative of $B$ with respect to $s$, the function
$B_{s}$ is only continuous on $[0,1)\times\mathbb{R}^{m}\times\partial\Delta$
and we have
\[
B_{s}(0,\rho,x)=0\text{ and }B_{s}(1,\rho,x)=\infty.
\]

The function $B_{ss}$, like $B_{s}$, is only continuous on $[0,1)\times\mathbb{R}^{m}\times\partial\Delta$
and we have
\[
B_{ss}(0,\rho,x)=\sum_{i=0}^{m}\frac{(x_{i}-\mu_{i})^{2}}{\mu_{i}}\text{ and }B_{ss}(1,\rho,x)=\infty.
\]

In contrast, the function $B_{\rho_{p}}$ can be extended continuously
to $[0,1]\times\mathbb{R}^{m}\times\partial\Delta$ (from $[0,1)\times\mathbb{R}^{m}\times\partial\Delta$)
and we have
\[
B_{\rho_{p}}(0,\rho,x)=0\text{ and }B_{\rho_{p}}(1,\rho,x)=-[x_{p}-\mu_{p}(\rho)].
\]

Finally, the function $B_{\rho_{p}\rho_{q}}$ can also be extended
continuously to $[0,1]\times\mathbb{R}^{m}\times\partial\Delta$ (from
$[0,1)\times\mathbb{R}^{m}\times\partial\Delta$) and we have
\[
B_{\rho_{p}\rho_{q}}(0,\rho,x)=0\text{ and }B_{\rho_{p}\rho_{q}}(1,\rho,x)=\mu_{p,q}(p).
\]

\subsection{Properties of the function $h$ }

Having detailed the essential computations for the partial derivatives
of $B$, we return to the function $h$ and recall that
\begin{gather*}
h:H\twoheadrightarrow(0,1)\\
(t,\rho,x)\mapsto s\text{ such that }B(s,\rho,x)=t,
\end{gather*}
where $H=\{(t,\rho,x)\in(0,\infty)\times\mathbb{R}^{m}\times\partial\Delta\mid0<t<b(x,\rho)\}$.
Differentiating the equation $B\left(h(t,\rho,x),\rho,x\right)=t$
implicitly in $t$ and $\rho_{p}$, we obtain
\begin{equation}
h_{t}(t,\rho,x)=\frac{1}{B_{s}(h,\rho,x)}>0,\label{eq:ht}
\end{equation}
and
\begin{equation}
h_{\rho_{p}}(t,\rho,x)=-\frac{B_{\rho_{p}}(h,\rho,x)}{B_{s}(h,\rho,x)}.\label{eq:hrhop}
\end{equation}
We recall these partial derivatives are continuous on $H$ and look
to investigate their behavior on the boundary of $H$. In section
\ref{sub:Alongh1}, we will look at the behavior as $(t,\rho,x)$
tends to $(0,\tilde{\rho},\tilde{x})$; in section \ref{sub:Alongh2},
we will look at the behavior as $(t,\rho,x)$ tends to $(b(\tilde{x},\tilde{\rho}),\tilde{\rho},\tilde{x})$;
and in section \ref{sub:Summarizeh}, we will summarize these computations.

\subsubsection{Along $(0,\rho,x)$.\label{sub:Alongh1}}

Fix $(\tilde{\rho},\tilde{x})\in\mathbb{R}^{n}\times\partial\Delta$.
Also consider the points $(t,\rho,x)\in H$.

For the limit of $h$ as $(t,\rho,x)$ tends to $(0,\tilde{\rho},\tilde{x})$,
we keep in mind that $h$ is an increasing function with respect to
$t$. Let $\varepsilon>0$, then there exists $\delta_{1}>0$ such
that $\left|h(\delta_{1},\tilde{\rho},\tilde{x})\right|<\varepsilon/2$
and $\delta_{1}<\tilde{t}/2$. By the continuity of $h$ on the interior
of $H$, there exists $\delta_{2}>0$ such that $\left|h(\delta_{1},\tilde{\rho},\tilde{x})-h(\delta_{1},\rho,x)\right|<\varepsilon/2$
for all $\left\Vert (\rho,x)-(\tilde{\rho},\tilde{x})\right\Vert <\delta_{2}$.
Then let $\delta=\min(\delta_{1},\delta_{2})$ so that $\left\Vert (t,\rho,x)-(0,\tilde{\rho},\tilde{x})\right\Vert <\delta$
implies
\begin{alignat*}{1}
0\leq h(t,\rho,x) & <h(\delta,\rho,x)\\
 & \leq h(\delta_{1},\rho,x)\\
 & =h(\delta_{1},\rho,x)-h(\delta_{1},\tilde{\rho},\tilde{x})+h(\delta_{1},\tilde{\rho},\tilde{x})\\
 & <\frac{\varepsilon}{2}+\frac{\varepsilon}{2}=\varepsilon.
\end{alignat*}
Thus we have
\begin{equation}
\lim_{(t,\rho,x)\to(0,\tilde{\rho},\tilde{x})}h(t,\rho,x)=0.\label{eq:lim0h}
\end{equation}

For the next limit, we look at (\ref{eq:ht}) in conjunction with
(\ref{eq:lim0Bs}) and (\ref{eq:lim0h}) and immediately obtain
\[
\lim_{(t,\rho,x)\to(0,\tilde{\rho},\tilde{x})}h_{t}(t,\rho,x)=\infty.
\]

For the limit of $h_{\rho_{p}}$ as $(t,\rho,x)$ tends to $(0,\tilde{\rho},\tilde{x})$,
we apply the mean value theorem and imitate the idea of l'Hospital's
rule. By the mean value theorem, we have
\begin{alignat*}{1}
h_{\rho_{p}}(t,\rho,x) & =-\frac{B_{\rho_{p}}(h,\rho,x)}{B_{s}(h,\rho,x)}\cdot\frac{\frac{1}{h}}{\frac{1}{h}}\\
 & =-\frac{B_{\rho_{p}s}(\xi_{1},\rho,x)}{B_{ss}(\xi_{2},\rho,x)}
\end{alignat*}
for some $\xi_{1},\xi_{2}\in(0,h)$. Together with (\ref{eq:lim0Bss}),
(\ref{eq:lim0Brhops}), and (\ref{eq:lim0h}), we obtain
\begin{equation}
\lim_{(t,\rho,x)\to(0,\tilde{\rho},\tilde{x})}h_{\rho_{p}}(t,\rho,x)=0.\label{eq:lim0hrhop}
\end{equation}

Finally, we look at the limit of $h\cdot h_{t}$ as $(t,\rho,x)$
tends to $(0,\tilde{\rho},\tilde{x})$. As with the proof of \ref{eq:lim0hrhop},
we apply the mean value theorem and imitate the idea of l'Hospital's
rule. By the mean value theorem, we have
\begin{alignat*}{1}
h(t,\rho,x)h_{t}(t,\rho,x) & \overset{\ref{eq:ht}}{=}\frac{1}{\frac{1}{h}B_{s}(h,\rho,x)}\\
 & =\frac{1}{B_{ss}(\xi,\rho,x)}
\end{alignat*}
for some $\xi\in(0,h)$. By a combination of (\ref{eq:lim0h}) and
(\ref{eq:lim0Bss}), we have
\[
\lim_{(t,\rho,x)\to(0,\tilde{\rho},\tilde{x})}h(t,\rho,x)h_{t}(t,\rho,x)=\left[\sum_{i=0}^{m}\frac{\left[x_{i}-\mu_{i}(\tilde{\rho})\right]^{2}}{\mu_{i}(\tilde{\rho})}\right]^{-1}.
\]
Thus, while $h_{t}$ tends to $\infty$ as $(t,\rho,x)$ tends to
$(0,\tilde{\rho},\tilde{x})$, the product of $h$ with $h_{t}$ has
a finite limit as $(t,\rho,x)$ tends to $(0,\tilde{\rho},\tilde{x})$.

\subsubsection{Along $(b(x,\rho),\rho,x)$.\label{sub:Alongh2}}

Fix $(\tilde{\rho},\tilde{x})\in\mathbb{R}^{n}\times\partial\Delta$
and $\tilde{t}=b(\tilde{x},\tilde{\rho})$. Also consider the points
$(t,\rho,x)\in H$.

Let us start by determining the limit of $h$ as $(t,\rho,x)$ tends
to $(\tilde{t},\tilde{\rho},\tilde{x})$. As with the proof of (\ref{eq:lim0h}),
we keep in mind that $h$ is an increasing function with respect to
$t$. Let $\varepsilon>0$, then there exists $\delta_{1}>0$ such
that $\tilde{t}-\delta_{1}>0$ and $[1-h(\tilde{t}-\delta_{1},\tilde{\rho},\tilde{x})]<\varepsilon/2$.
We also have $\tilde{t}-\delta_{1}<\tilde{t}=b(\tilde{x},\tilde{\rho})$
and choose, by the continuity of $b$, a $\delta_{2}>0$ so that $\left|b(x,\rho)-b(\tilde{x},\tilde{\rho})\right|<\delta_{1}$
for all $\left\Vert (\rho,x)-(\tilde{\rho},\tilde{x})\right\Vert <\delta_{2}$.
This ensures $b(x,\rho)>b(\tilde{x},\tilde{\rho})-\delta_{1}=\tilde{t}-\delta_{1}$
for all $\left\Vert (t,\rho,x)-(\tilde{t},\tilde{\rho},\tilde{x})\right\Vert <\delta_{2}$
and consequently $(\tilde{t}-\delta_{1},\rho,x)\in H$. Then by the
continuity of $h$ on $H$, there exists $\delta_{3}>0$ (chosen to
be less than $\delta_{2}$) such that $\left|h(\tilde{t}-\delta_{1},\tilde{\rho},\tilde{x})-h(\tilde{t}-\delta_{1},\rho,x)\right|<\varepsilon/2$
for all $\left\Vert (\rho,x)-(\tilde{\rho},\tilde{x})\right\Vert <\delta_{3}$.
Then let $\delta=\min(\delta_{1},\delta_{3})$ so that $\left\Vert (t,\rho,x)-(\tilde{t},\tilde{\rho},\tilde{x})\right\Vert <\delta$
implies 
\begin{alignat*}{1}
0\leq1-h(t,\rho,x) & <1-h(\tilde{t}-\delta,\rho,x)\\
 & \leq1-h(\tilde{t}-\delta_{1},\rho,x)\\
 & =1-h(\tilde{t}-\delta_{1},\tilde{\rho},\tilde{x})+h(\tilde{t}-\delta_{1},\tilde{\rho},\tilde{x})-h(\tilde{t}-\delta_{1},\rho,x)\\
 & <\frac{\varepsilon}{2}+\frac{\varepsilon}{2}=\varepsilon.
\end{alignat*}
Thus we have
\begin{equation}
\lim_{(t,\rho,x)\to(\tilde{t},\tilde{\rho},\tilde{x})}h(t,\rho,x)=1.\label{eq:limth}
\end{equation}

Next, we note that (\ref{eq:lim1Bs}) and (\ref{eq:limth}) imply
\[
\lim_{(t,\rho,x)\to(\tilde{t},\tilde{\rho},\tilde{x})}h_{t}(t,\rho,x)=0.
\]

Finally, we combine (\ref{eq:lim1Bs}), (\ref{eq:lim1Brhop}), and
(\ref{eq:limth}) with (\ref{eq:hrhop}) to conclude
\[
\lim_{(t,\rho,x)\to(\tilde{t},\tilde{\rho},\tilde{x})}h_{\rho_{p}}(t,\rho,x)=0.
\]

\subsubsection{Summary for $h$\label{sub:Summarizeh}}

From the work shown above, we conclude that $h$ can be extended continuously
to $[0,\infty)\times\mathbb{R}^{m}\times\partial\Delta$ (from $(0,b(x,\rho))\times\mathbb{R}^{m}\times\partial\Delta$).
Explicitly, we have

\[
h(t,\rho,x)=\begin{cases}
0, & \text{if }t=0\\
B(h,\rho,x)=t, & \text{if }0\leq t\leq b(x,\rho)\\
1, & \text{if }t\geq b(x,\rho).
\end{cases}
\]

The partial derivative $h_{t}$ can only be extended continuously
to $(0,\infty)\times\mathbb{R}^{n}\times\partial\Delta$ (from $(0,b(x,\rho))\times\mathbb{R}^{n}\times\partial\Delta$)
and it tends towards infinity as $(t,\rho,x)$ tends to $(0,\tilde{\rho},\tilde{x})$
for any $(\tilde{\rho},\tilde{x})\in\mathbb{R}^{m}\times\partial\Delta$.
We have
\begin{alignat*}{1}
h_{t}(t,\rho,x)=\begin{cases}
\dfrac{1}{B_{s}(h\rho,x)}, & \text{if }0<t\leq b(x,\rho)\\
0, & \text{if }t\geq b(x,\rho).
\end{cases}
\end{alignat*}

Of slightly greater importance, we can continuously extend the product
$h\cdot h_{t}$ to $[0,\infty)\times\mathbb{R}^{m}\times\partial\Delta$
(from $(0,b(x,\rho))\times\mathbb{R}^{m}\times\partial\Delta$) and
obtain
\[
\left(h\cdot h_{t}\right)(t,\rho,x)=\begin{cases}
\left[\sum_{i=0}^{n}\dfrac{\left[x_{i}-\mu_{i}\right]^{2}}{\mu_{i}}\right]^{-1}, & \text{if }t=0\\
\dfrac{h}{B_{s}(h,\rho,x)}, & \text{if }0\leq t\leq b(x,\rho)\\
0, & \text{if }t\geq b(x,\rho).
\end{cases}
\]

Finally, the partial derivative $h_{\rho_{p}}$ can be extended continuously
to $[0,\infty)\times\mathbb{R}^{m}\times\partial\Delta$ (from $(0,b(x,\rho))\times\mathbb{R}^{m}\times\partial\Delta$).
Explicitly, we have
\begin{alignat*}{1}
h_{\rho_{p}}(t,\rho,x) & =\begin{cases}
0, & \text{if }t=0\\
-\dfrac{B_{\rho_{p}}(h,\rho,x)}{B_{s}(h,\rho,x)}, & \text{if }0\leq t\leq b(x,\rho)\\
0, & \text{if }t\geq b(x,\rho).
\end{cases}
\end{alignat*}

\section{\label{sec:Proof-of-the-Theorem}Proof of the theorem}

By (\ref{eq:main3}), the proof of Theorem \ref{thm:maintheorem}
reduces to showing 
\[
G_{n}(\rho)\doteq\int_{0}^{\infty}[1-D(t,\rho)]^{n}dt
\]
is twice differentiably continuous on $\mathbb{R}^{m}$. Working with
only fixed $n$ and wishing to use subscript notation when taking
partial derivatives, we henceforth omit the subscript $n$ and simply
write $G$ instead of $G_{n}$.

In order to differentiate $G$, we first look to express $D$ as an
integral (Lemma \ref{lem:D} in Section \ref{sub:RewriteD}). The
proof of Theorem \ref{thm:maintheorem} proceeds and concludes by
qualifying various applications of Lebesgue's dominated convergence
theorem.

\subsection{\label{sub:RewriteD}Rewriting the distribution function $D(t,\rho)$
as an integral}

Before expressing $D$ as an integral, let us consider the volume
form $d\lambda=m!d\lambda_{1}\wedge\dots\wedge d\lambda_{m}$ on the
unit simplex $\Delta$. Fix $\rho\in\mathbb{R}^{m}$ and for $i=0,1,\dots,m$,
we define $R_{i}(\rho)\subset\Delta$ as the region between the point
$\mu(\rho)\in\Delta^{\circ}$ and the $i$-th facet $\partial^{i}\Delta$
of $\Delta$. Explicitly, the set 
\begin{equation}
R_{i}(\rho)\doteq\{\lambda\in\Delta\mid\lambda=(1-s)\mu(\rho)+sx,\text{ for some }x\in\partial^{i}\Delta\text{ and }0\leq s\leq1\}.\label{eq:Rk}
\end{equation}
Together with the above notation, we consider $d\lambda$ under a
change of variables given by the map $\Lambda$ from Section \ref{sec:Notation-and-Properties}.
\begin{lem}
\label{lem:changeOfVariables}Define the map 
\begin{gather*}
\Lambda:[0,1]\times\mathbb{R}^{m}\times\partial\Delta\twoheadrightarrow\Delta\\
(s,\rho,x)\mapsto(1-s)\mu(\rho)+sx
\end{gather*}
and fix a $\rho\in\mathbb{R}^{m}$. Then the function $\lambda(s,x)=\Lambda(s,\rho,x)$
restricted to $[0,1]\times\partial^{i}\Delta$ is one-to-one and onto
$R_{i}(\rho)\subset\Delta$. We also have 
\begin{equation}
d\lambda|_{R_{i}}=\mu_{i}\cdot ms^{m-1}ds\wedge dx_{(i)},\label{eq:changeofvariableslemma}
\end{equation}
where $dx_{(i)}\doteq-(m-1)!\left(\frac{\partial}{\partial x_{i}}\lrcorner dx\right)$
and $\frac{\partial}{\partial x_{0}}\doteq-\frac{1}{m}\sum_{i=1}^{m}\frac{\partial}{\partial x_{i}}$;
geometrically, the vector $\frac{\partial}{\partial x_{i}}$ is the
inward facing normal to $\partial^{i}\Delta$ for $i=0,1,\dotsc,m$. \end{lem}
\begin{proof}
The first part of the lemma is immediate from the surjectivity of
$\Lambda$ onto $\Delta$, the convexity of $\Delta$, and the definition
of $R_{i}(\rho)$. For the second part, let $\lambda(s,x)=(1-s)\mu+sx$
so that $d\lambda_{i}=(x_{i}-\mu_{i})ds+sdx_{i}$ and
\begin{alignat*}{1}
d\lambda_{1}\wedge\dotsb\wedge d\lambda_{m} & =\sum_{i=1}^{m}(x_{i}-\mu_{i})s^{m-1}dx_{1}\wedge\dotsb\wedge dx_{i-1}\wedge ds\wedge d_{i+1}\wedge\dotsb\wedge dx_{m}\\
 & =\sum_{i=1}^{m}(x_{i}-\mu_{i})s^{m-1}(-1)^{i-1}ds\wedge dx_{1}\wedge\dotsb\widehat{dx_{i}}\dotsb\wedge dx_{m}\\
 & =\sum_{i=1}^{m}(x_{i}-\mu_{i})s^{m-1}ds\wedge\left(\frac{\partial}{\partial x_{i}}\lrcorner dx\right).
\end{alignat*}
Fix $k\in\{1,\dots,m\}$. Then for any $x\in\partial^{k}\Delta$ we
have $x_{k}=0$ and this implies $dx_{k}=0$. In this case,
\[
\begin{alignedat}{1}\sum_{i=1}^{m}(x_{i}-\mu_{i})s^{m-1}ds\wedge\left(\frac{\partial}{\partial x_{i}}\lrcorner dx\right) & =-\mu_{k}\cdot s^{m-1}ds\wedge\left(\frac{\partial}{\partial x_{k}}\lrcorner dx\right)\end{alignedat}
\]
as claimed above. Now let $k=0$. Then for any $x\in\partial^{0}\Delta$
we have $\sum_{i=1}^{m}x_{i}=1$ and this implies $\sum_{i=1}^{m}dx_{i}=0$.
From this we obtain 
\[
\begin{alignedat}{1}\frac{\partial}{\partial x_{j}}\lrcorner dx & =(-1)^{j-1}dx_{1}\wedge\dotsb\widehat{dx_{j}}\dotsb\wedge dx_{m-1}\wedge dx_{m}\\
 & =(-1)^{j-1}dx_{1}\wedge\dotsb\widehat{dx_{j}}\dotsb\wedge dx_{m-1}\wedge(-\sum_{i=1}^{m-1}dx_{i})\\
 & =(-1)^{j}dx_{1}\wedge\dotsb\widehat{dx_{j}}\dotsb\wedge dx_{m-1}\wedge(dx_{j})\\
 & =(-1)^{j}(-1)^{m-j-1}dx_{1}\wedge\dotsb\wedge dx_{m-1}\\
 & =\frac{\partial}{\partial x_{m}}\lrcorner dx
\end{alignedat}
\]
for any $j=1,\dotsc,m-1$ and hence
\[
-\frac{\partial}{\partial x_{0}}\lrcorner dx=\frac{\partial}{\partial x_{j}}\lrcorner dx
\]
for any $j=1,\dotsc,m$. Thus we have
\[
\begin{alignedat}{1}\sum_{i=1}^{m}(x_{i}-\mu_{i})s^{m-1}ds\wedge\left(\frac{\partial}{\partial x_{i}}\lrcorner dx\right) & =\sum_{i=1}^{m}(x_{i}-\mu_{i})s^{m-1}ds\wedge\left(-\frac{\partial}{\partial x_{0}}\lrcorner dx\right)\\
 & =-(x_{0}-\mu_{0})s^{m-1}ds\wedge\left(-\frac{\partial}{\partial x_{0}}\lrcorner dx\right)\\
 & =-\mu_{0}\cdot s^{m-1}ds\wedge\left(\frac{\partial}{\partial x_{0}}\lrcorner dx\right)
\end{alignedat}
\]
as claimed above. For an alternative proof, rewrite $d\lambda_{m}=(x_{m}-\mu_{m})ds+sdx_{m}=(x_{m}-\mu_{m})ds-sdx_{1}-\dotsb-sdx_{m-1}$
and so $d\lambda_{1}\wedge\dotsb\wedge d\lambda_{m}=\det(M)ds\wedge dx_{1}\wedge\dotsb\wedge dx_{m-1}$,
where 
\[
M=\left[\begin{array}{ccccc}
(x_{1}-\mu_{1}) & s & 0 & \cdots & 0\\
(x_{2}-\mu_{2}) & 0 & s & \cdots & 0\\
\vdots & \vdots & \vdots & \ddots & \vdots\\
(x_{m-1}-\mu_{m-1}) & 0 & 0 & \cdots & s\\
(x_{m}-\mu_{m}) & -s & -s & \cdots & -s
\end{array}\right].
\]
By switching rows and row reducing, we obtain
\begin{alignat*}{1}
\det M & =(-1)^{m-1}\det\left[\begin{array}{ccccc}
(x_{m}-\mu_{m}) & -s & -s & \cdots & -s\\
(x_{1}-\mu_{1}) & s & 0 & \cdots & 0\\
(x_{2}-\mu_{2}) & 0 & s & \cdots & 0\\
\vdots & \vdots & \vdots & \ddots & \vdots\\
(x_{m-1}-\mu_{m-1}) & 0 & 0 & \cdots & s
\end{array}\right]\\
 & =(-1)^{m-1}\det\left[\begin{array}{ccccc}
(\sum_{i=1}^{m}x_{i}-\sum_{i=1}^{m}\mu_{i}) & 0 & 0 & \cdots & 0\\
(x_{1}-\mu_{1}) & s & 0 & \cdots & 0\\
(x_{2}-\mu_{2}) & 0 & s & \cdots & 0\\
\vdots & \vdots & \vdots & \ddots & \vdots\\
(x_{m-1}-\mu_{m-1}) & 0 & 0 & \cdots & s
\end{array}\right]\\
 & =(-1)^{m-1}\mu_{0}s^{n-1}.
\end{alignat*}

\end{proof}
Having proven Lemma \ref{lem:changeOfVariables}, we are ready to
take the vague notion of writing $D$ as an integral given in Section
\ref{sub:Motivation} and make it rigorous.
\begin{lem}
\label{lem:D}The function $D(t,\rho)=\mathcal{P}\{\lambda\in\Delta\mid b(\lambda,\rho)\leq t\}$
can be expressed as an integral. We have 
\begin{equation}
D(t,\rho)=\sum_{i=0}^{m}\mu_{i}\cdot\int_{\partial^{i}\Delta}h^{m}(t,\rho,x)\, dx_{(i)},\label{eq:DinLemmaD}
\end{equation}
where $dx_{(i)}\doteq-(n-1)!\left(\frac{\partial}{\partial x_{i}}\lrcorner dx\right)$
and $\frac{\partial}{\partial x_{0}}\doteq-\frac{1}{m}\sum_{i=1}^{m}\frac{\partial}{\partial x_{i}}$.
Moreover, the function $D$ is continuous on $[0,\infty)\times\mathbb{R}^{m}$.\end{lem}
\begin{proof}
Let $S=S(t,\rho)\doteq\{\lambda\in\Delta\mid b(\lambda,\rho)\leq t\}$
so that $D=\int_{S}d\lambda$ and define $S_{i}\doteq S\cap R_{i}$,
where $R_{i}$ are defined at (\ref{eq:Rk}). It follows immediately
that $S=\cup_{i=0}^{m}S_{i}$ and $\mathcal{P}(S_{i}\cap S_{j})=0$
for $i\neq j$. Thus far, we have
\[
D(t,\rho)=\sum_{i=0}^{m}\int_{S_{i}(t,\rho)}d\lambda
\]
and proceed by giving $S_{i}(t,\rho)$ explicitly. We claim that
\begin{equation}
S_{i}(t,\rho)=\{\lambda\in\Delta\mid\lambda=(1-s)\mu(\rho)+sx\text{ for some }x\in\partial^{i}\Delta\text{ and }0\leq s\leq h(t,\rho,x)\}.\label{eq:Si}
\end{equation}

Starting with the forward inclusion, we take $\lambda\in S_{i}=S\cap R_{i}$.
Then $b(\lambda,\rho)\leq t$ and $\lambda=(1-s)\mu(\rho)+sx$ for
some $x\in\partial^{i}\Delta$ and some $s\in[0,1]$. These conditions
imply that $B(s,\rho,x)=b(\lambda,\rho)\leq t$. Then since $h(\cdot,\rho,x)$
is an increasing function and inverse to $B(\cdot,\rho,x)$, we obtain
$s=h\bigl(B(s,\rho,x),\rho,x\bigr)\leq h(t,\rho,x)$ and this gives
the desired inclusion. For the reverse inclusion, let $\lambda=(1-s)\mu(\rho)+sx$
for some $x\in\partial^{i}\Delta$ and $0\leq s\leq h(t,\rho,x)$.
Then because $B(\cdot,\rho,x)$ is an increasing function and inverse
to $h(\cdot,\rho,x)$ we have $b(\lambda,\rho)=B(s,\rho,x)\leq B\bigl(h(t,\rho,x),\rho,x\bigr)=t$.
This concludes the reverse inclusion and proves (\ref{eq:Si}). 

We next make the change of variables $\lambda(s,x)=(1-s)\mu(\rho)+sx$.
Using Lemma \ref{lem:changeOfVariables}, we obtain
\begin{alignat*}{1}
\sum_{i=0}^{m}\int_{S_{i}(t,\rho)}d\lambda & \overset{\eqref{eq:changeofvariableslemma}}{=}\sum_{i=0}^{m}\int_{S_{i}(t,\rho)}\mu_{i}\cdot ms^{m-1}ds\wedge dx_{(i)}\\
 & =\sum_{i=0}^{m}\mu_{i}\cdot\int_{\partial^{i}\Delta}\left[\int_{0}^{h(t,\rho,x)}ms^{m-1}ds\right]dx_{(i)}\\
 & =\sum_{i=0}^{m}\mu_{i}\cdot\int_{\partial^{i}\Delta}h^{m}(t,\rho,x)\, dx_{(i)}.
\end{alignat*}
Note that we can apply Fubini's theorem in the second equality, because
$0\le D\leq1$ and hence the integral is finite.

The continuity of $D$ follows from Lebesgue's dominated convergence
theorem.\end{proof}
\begin{rem}
For $m=1$, the formulation (\ref{eq:DinLemmaD}) is the integration
of a $0$-form over a $0$-chain, which is evaluation at a point.
Explicitly, we have
\begin{alignat}{1}
D(t,\rho) & =\sum_{i=0}^{1}\mu_{i}\cdot\int_{\partial^{i}\Delta}h(t,\rho,x)\, dx_{(i)}\label{eq:Dwrth}\\
 & =\sum_{i=0}^{1}\mu_{i}\cdot h(t,\rho,\partial^{i}\Delta)\nonumber \\
 & =\mu_{0}\cdot h(t,\rho,\partial^{0}\Delta)+\mu_{1}\cdot h(t,\rho,\partial^{1}\Delta)\nonumber \\
 & =\mu_{0}\cdot h(t,\rho,1)+\mu_{1}\cdot h(t,\rho,0),\nonumber 
\end{alignat}
which agrees with (\ref{eq:Dwrtf}). 
\end{rem}

\subsection{The first partial derivative of $G$}

Now we look to differentiate our function $G$, bring the derivative
under the integral sign, and conclude the derivative is continuous.
\begin{lem}
\label{lem:Drhop}Given the function 
\[
D(t,\rho)=\sum_{i=0}^{m}\mu_{i}\cdot\int_{\partial^{i}\Delta}h^{m}(t,\rho,x)\, dx_{(i)}
\]
 on $[0,\infty)\times\mathbb{R}^{m}$, the partial derivative $D_{\rho_{p}}$
exists and is continuous on $[0,\infty)\times\mathbb{R}^{m}$. It
is given by
\[
D_{\rho_{p}}=\sum_{i=0}^{m}\mu_{i}\cdot\int_{\partial^{i}\Delta}mh^{m-1}h_{\rho_{p}}\, dx_{(i)}+\mu_{i,p}\cdot\int_{\partial^{i}\Delta}h^{m}\, dx_{(i)}.
\]

Consequently, the partial derivative of 
\[
G(\rho)=\int_{0}^{\infty}[1-D(t,\rho)]^{n}dt
\]
 exists and is continuous on $\mathbb{R}^{m}$. It is given by
\[
G_{\rho_{p}}=-\int_{0}^{\infty}n[1-D]^{n-1}D_{\rho_{p}}dt.
\]
\end{lem}
\begin{proof}
Referring to Section \ref{sub:Summarizeh}, the function $h^{m}$
and its partial derivative $mh^{m-1}h_{\rho_{p}}$ are continuous
on $[0,\infty)\times\mathbb{R}^{m}\times\partial\Delta$. Thus, by
Lebesgue's dominated convergence theorem, we may differentiate under
the integral sign and obtain 
\[
\frac{\partial}{\partial\rho_{p}}\int_{\partial^{i}\Delta}h^{m}dx_{(i)}=\int_{\partial^{i}\Delta}mh^{m-1}h_{\rho_{p}}dx_{(i)}.
\]
We note that the dominated convergence theorem also implies the right-hand
side is continuous on $[0,\infty)\times\mathbb{R}^{m}$.

The function $D$ is continuous on $[0,\infty)\times\mathbb{R}^{m}$
(Lemma \ref{lem:D}) and the function $D_{\rho_{p}}$ is continuous
on $[0,\infty)\times\mathbb{R}^{m}$ (first part of this lemma). Thus,
the function $\frac{\partial}{\partial\rho_{p}}[1-D]^{n}$ exists,
is given by $n[1-D]^{n-1}D_{\rho_{p}}$, and is continuous on $[0,\infty)\times\mathbb{R}^{m}$.
Next we let $\tilde{\rho}\in\mathbb{R}^{m}$, choose $r>0$, and consider
the support of the family of functions $t\mapsto1-D(t,\rho)$ with
$\rho\in B_{r}(\tilde{\rho})$. For a fixed $\rho\in B_{r}(\tilde{\rho})$,
the support of $t\mapsto1-D(t,\rho)$ is contained in the interval
$[0,\max_{\lambda\in\Delta}b(\lambda,\rho)]$. We recalling from Lemma
\ref{lem:functionb} that $\max_{\lambda\in\Delta}b(\lambda,\rho)=\max_{i\in\{0,1,\dots,m\}}\{\log(1+\left|e^{\rho}\right|)-\rho_{i}\}$.
For $i=0,1,\dotsc,m$, there exists $M_{i}>0$, depending on $\tilde{\rho}$
and $r$, such that $\log(1+\left|e^{\rho}\right|)-\rho_{i}<M_{i}$
for all $\rho\in B_{r}(\tilde{\rho})$. Consequently, there exists
an $M>0$, depending on $\tilde{\rho}$ and $r$, such that $\mathrm{supp}\bigl(1-D(\cdot,\rho)\bigr)\subset[0,M]$
for all $\rho\in B_{r}(\tilde{\rho})$. By Lebesgue\textquoteright{}s
dominated convergence theorem, we may differentiate under the integral
sign and obtain 
\begin{alignat*}{1}
G_{\rho_{p}}=\frac{\partial}{\partial\rho_{p}}\int_{0}^{\infty}[1-D]^{n}dt & =-\int_{0}^{\infty}n[1-D]^{n-1}D_{\rho_{p}}dt.
\end{alignat*}
The theorem also implies the function is continuous on $\mathbb{R}^{m}$.
\end{proof}

\subsection{The second partial derivative of $G$\label{sub:SecondPartialG}}

In Lemma \ref{lem:Drhop} we found that 
\[
D_{\rho_{p}}=\sum_{i=0}^{m}\mu_{i}\cdot\int_{\partial^{i}\Delta}mh^{m-1}h_{\rho_{p}}dx_{(i)}+\mu_{i,p}\cdot\int_{\partial^{i}\Delta}h^{m}dx_{(i)}
\]
is continuous on $[0,\infty)\times\mathbb{R}^{m}$. Then let $S=\sum_{i=0}^{m}\mu_{i}\cdot\int_{\partial^{i}\Delta}mh^{m-1}h_{\rho_{p}}dx_{(i)}$
and $\tilde{S}=\mu_{i,p}\cdot\int_{\partial^{i}\Delta}h^{m}dx_{(i)}$.
Thus, we have
\[
G_{\rho_{p}}=-\int_{0}^{\infty}n[1-D]^{n-1}Sdt-\int_{0}^{\infty}n[1-D]^{n-1}\tilde{S}dt.
\]
We first proceed with the second summand $-\int_{0}^{\infty}n[1-D]^{n-1}\tilde{S}dt$,
because it is easier to deal with, and then look at the first summand
$-\int_{0}^{\infty}n[1-D]^{n-1}Sdt$

\subsubsection{The summand $-\int_{0}^{\infty}n[1-D]^{n-1}\tilde{S}dt$ is continuously
differentiable}

The factor $\tilde{S}$ is continuous and has continuous derivatives
in $\rho_{p}$ over $[0,\infty)\times\mathbb{R}^{m}$ for the same
reason that $D$ has continuous derivatives in $\rho_{p}$ over $[0,\infty)\times\mathbb{R}^{m}$
(Lemma \ref{lem:D}, Lemma \ref{lem:Drhop}). Then the integrand of
the second summand $n[1-D]^{n-1}\tilde{S}$ and its derivative in
$\rho_{q}$, $n(n-1)[1-D]^{n-2}D_{\rho_{q}}\tilde{S}+n[1-D]^{n-1}\tilde{S}_{\rho_{q}}$,
are continuous over $[0,\infty)\times\mathbb{R}^{m}$. Imitating the
proof of Lemma \ref{lem:Drhop}, the first partial derivative of the
second summand with respect to $\rho_{q}$ exists, is given by
\[
\frac{\partial}{\partial\rho_{q}}\int_{0}^{\infty}n[1-D]^{n-1}\tilde{S}dt=\int_{0}^{\infty}\frac{\partial}{\partial\rho_{q}}\left[n[1-D]^{n-1}\tilde{S}\right]dt,
\]
and is continuous over $\mathbb{R}^{m}$.

\subsubsection{The summand $-\int_{0}^{\infty}n[1-D]^{n-1}Sdt$ is continuously
differentiable}

Now we are left to see that $\int_{0}^{\infty}n[1-D]^{n-1}Sdt$ has
a continuous partial derivative. We first write
\begin{alignat}{1}
\int_{0}^{\infty}n[1-D]^{n-1}Sdt & =\int_{0}^{\infty}n[1-D]^{n-1}\left[\sum_{i=0}^{m}\mu_{i}\cdot\int_{\partial^{i}\Delta}mh^{m-1}h_{\rho_{p}}dx_{(i)}\right]dt\label{eq:firstSUMMANDrewrite}\\
 & =nm\sum_{i=0}^{m}\mu_{i}\cdot\int_{0}^{\infty}\left[\int_{\partial^{i}\Delta}\left[1-D\right]^{n-1}h^{m-1}h_{\rho_{p}}dx_{(i)}\right]dt\nonumber \\
 & =nm\sum_{i=0}^{m}\mu_{i}\cdot T_{i},\nonumber 
\end{alignat}
 where 
\begin{equation}
T_{i}(\rho)\doteq\int_{0}^{\infty}\left[\int_{\partial^{i}\Delta}[1-D(t,\rho)]^{n-1}h^{m-1}(t,\rho,x)h_{\rho_{p}}(t,\rho,x)dx_{(i)}\right]dt.\label{eq:T_i}
\end{equation}
Applying Fubini's theorem and making the change of variables $t(s)=B(s,\rho,x)$,
we obtain
\begin{alignat*}{1}
T_{i} & =\int_{\partial^{i}\Delta}\left[\int_{0}^{b(x,\rho)}[1-D]^{n-1}h^{m-1}h_{\rho_{p}}dt\right]dx_{(i)}\\
 & =\int_{\partial^{i}\Delta}\left[\int_{0}^{1}[1-D(B,\rho)]^{n-1}s^{m-1}h_{\rho_{p}}(B,\rho,x)B_{s}ds\right]dx_{(i)}.
\end{alignat*}
We note that the change of variables is motivated by our desire to
replace a moving domain of integration with a fixed domain of integration
\cite{F}. Finally, since $h_{\rho_{p}}(B,\rho,x)B_{s}=-B_{\rho_{p}}$
(\ref{eq:hrhop}), we have
\begin{equation}
T_{i}(\rho)=\int_{\partial^{i}\Delta}\left[\int_{0}^{1}s^{m-1}[1-D(B,\rho)]^{n-1}(-B_{\rho_{p}})ds\right]dx_{(i)}.\label{eq:integralinstep4}
\end{equation}

Section (\ref{sub:SummarizeB}) and Lemma \ref{lem:D} imply the integrand
$s^{m-1}[1-D(B,\rho)]^{n-1}(-B_{\rho_{p}})$ is a continuous function
on $[0,1]\times\mathbb{R}^{n}\times\partial\Delta$ . It's partial
derivative in $\rho_{q}$ is
\begin{multline}
s^{m-1}(n-1)[1-D(B,\rho)]^{n-2}[D_{t}(B,\rho)B_{\rho_{q}}+D_{\rho_{q}}(B,\rho)]B_{\rho_{p}}-s^{m-1}[1-D(B,\rho)]^{n-1}B_{\rho_{p}\rho_{q}}.\label{eq:integrand}
\end{multline}

Except for $(s,\rho,x)\mapsto D_{t}(B(s,\rho,x),\rho)$, all the functions
appearing in (\ref{eq:integrand}) have been shown to be continuous
on $[0,1]\times\mathbb{R}^{m}\times\partial\Delta$. More specifically,
the continuity of $D$ is given by Lemma \ref{lem:D}, the continuity
of $D_{\rho_{q}}$ is given by Lemma \ref{lem:Drhop}, and the continuity
of $B$, $B_{\rho_{p}}$, and $B_{\rho_{p}\rho_{k}}$ are stated in
Section \ref{sub:SummarizeB}. Thus, in order to show \ref{eq:integrand}
is continuous on $[0,1]\times\mathbb{R}^{m}\times\partial\Delta$,
we would like to show $D_{t}$ is continuous on $[0,\infty)\times\mathbb{R}^{m}$.
However, we only have the following lemma: 
\begin{lem}
\label{lem:Dt}For $m>1$, the partial derivative $D_{t}$ exists
and is continuous on $[0,\infty)\times\mathbb{R}^{m}$. It is given
by
\[
D_{t}=\sum_{i=0}^{m}\mu_{i}\cdot\int_{\partial^{i}\Delta}mh^{m-1}h_{t}\, dx_{(i)}\geq0.
\]

For $m=1$, the partial derivative $D_{t}$ exists and is continuous
on $(0,\infty)\times\mathbb{R}$. It is given by 
\begin{alignat}{1}
D_{t}(t,\rho) & =\sum_{i=0}^{1}\mu_{i}\cdot h_{t}(t,\rho,\partial^{i}\Delta).\label{eq:Dt}
\end{alignat}
\end{lem}
\begin{proof}
As stated in Section \ref{sub:Summarizeh}, the functions $h$ and
$h\cdot h_{t}$ are continuous on $[0,\infty)\times\mathbb{R}^{m}\times\partial\Delta$.
Thus the function $h^{m}$ and its partial derivative $mh^{m-1}h_{t}$
(for $m\geq2$) are continuous on $[0,\infty)\times\mathbb{R}^{m}\times\partial\Delta$.
By Lebesgue's dominated convergence theorem, we differentiate under
the integral sign and obtain 
\[
\frac{\partial}{\partial t}\int_{\partial^{i}\Delta}h^{m}dx_{(i)}=\int_{\partial^{i}\Delta}mh^{m-1}h_{t}dx_{(i)}.
\]
The theorem also gives the continuity of the function on $[0,\infty)\times\mathbb{R}^{m}$.

While we can take $m=1$ as a special case of $m>1$ except on $(0,\infty)\times\mathbb{R}^{m}\times\partial\Delta$,
we instead choose to differentiate the expression for $D$ given by
(\ref{eq:Dwrth}) with respect to $t$ and refer to the continuity
of $h_{t}$ as stated in Section \ref{sub:Summarizeh}.
\end{proof}
Thus, the function $D_{t}$ is only continuous on $[0,\infty)\times\mathbb{R}^{m}$
for $m>1$. While the function $D_{t}(B(s,\rho,x),\rho)$ tends to
$\infty$ as $(s,\rho,x)$ tends to $(0,\tilde{\rho},\tilde{x})$
for any $(\tilde{\rho},\tilde{x})\in\mathbb{R}^{m}\times\partial\Delta$,
the function $B_{\rho}(s,\rho,x)$ tends to $0$ as $(s,\rho,x)$
tends to $(0,\tilde{\rho},\tilde{x})$. It would be convenient then,
for $m=1$, if the map $(s,\rho,x)\mapsto D_{t}\bigl(B(s,\rho,x),\rho\bigr)B_{\rho}(s,\rho,x)$
could be extended continuously to $[0,1]\times\mathbb{R}\times\partial\Delta$.
Since Section (\ref{sub:SummarizeB}) and Lemma \ref{lem:Dt} imply
it is continuous on $(0,1]\times\mathbb{R}\times\partial\Delta$,
it suffices to show the limit as $(s,\rho,x)$ tends to $(0,\tilde{\rho},\tilde{x})$
for any $(\tilde{\rho},\tilde{x})\in\mathbb{R}\times\partial\Delta$
exists.

However, before setting out to do this, we shall prove a technical
lemma concerning the function $B$. The first part of the lemma essentially
says that since $s\mapsto B(s,\rho,0)$ and $s\mapsto B(s,\rho,1)$
are strictly increasing functions starting at $0$ for $s=0$, for
sufficiently small $s$, we can find $\tilde{s}$ such that $B(s,\rho,0)=B(\tilde{s},\rho,1)$
and vice versa. The second part of the lemma essentially says that
since $B(\cdot,\rho,0)$ and $B(\cdot,\rho,1)$ are shaped similarly,
the rate at which $\tilde{s}$ goes to $0$ is almost proportional
to the rate at which $s$ goes to $0$. Rigorously, we have the following:
\begin{lem}
\label{lem:stilde}There exists $\delta>0$ such that for any $s\in[0,\delta]$,
there exists $\tilde{s}=\tilde{s}(s,\rho)$ such that 
\begin{equation}
B(s,\rho,x)=B(\tilde{s},\rho,1-x)\label{eq:equality}
\end{equation}

Furthermore, we have 
\[
\frac{d\tilde{s}}{ds}(0,\rho)=\lim_{s\to0}\frac{\tilde{s}}{s}=\exp[(-1)^{x}\rho].
\]
\end{lem}
\begin{proof}
For the first part, we recall from Section \ref{sub:Motivation} that
the function $b(\cdot,\rho)$ is convex with a minimum at $\mu(\rho)$
so its inverse with two branches, $g_{-1}$ and $g_{1}$, given by
\begin{alignat*}{1}
g_{-1}(t,\rho) & =\lambda\quad\text{if }t\in[0,b(0,\rho)],\, b(\lambda,\rho)=t,\text{ and }\lambda\in[0,\mu(\rho)]
\end{alignat*}
and
\begin{alignat*}{1}
g_{1}(t,\rho) & =\lambda\quad\text{if }t\in[0,b(1,\rho)],\, b(\lambda,\rho)=t,\text{ and }\lambda\in[\mu(\rho),1].
\end{alignat*}
Thus, for $t\in[0,\min\{b(0,\rho),b(1,\rho)\}]$, we have $g_{-1}(t,\rho)\in[0,\mu(\rho)]$
and $g_{1}(t,\rho)\in[\mu(\rho),1]$ with $b(g_{-1},\rho)=t=b(g_{1},\rho)$.
Separately, given $\lambda_{-1}\in[0,\mu(\rho)]$ and $\lambda_{1}\in[\mu(\rho),1]$,
there exists $s_{-1}\geq0$ and $s_{1}\geq0$ such that $\lambda_{-1}=(1-s_{-1})\mu+s_{-1}\cdot0$
and $\lambda_{1}=(1-s_{1})\mu+s_{1}\cdot1$. Letting $\varepsilon=\min\{b(0,\rho),b(1,\rho)\}$,
there is a $\delta>0$ such that $0\leq B(s,\rho,x)<\varepsilon$
for all $s\in[0,\delta]$. Then for any $s\in[0,\delta]$, we have
$B(s,\rho,0)\in[0,\varepsilon]$, so $g_{1}(B(s,\rho,0),\rho)$ lies
in the interval $[\mu(\rho),1]$, and we find $\tilde{s}\geq0$ such
that 
\[
g_{1}(B(s,\rho,0),\rho)=(1-\tilde{s})\mu+\tilde{s}\cdot1.
\]
Taking $b(\cdot,\rho)$ of both sides, we obtain 
\begin{alignat*}{1}
B(s,\rho,0) & =b((1-\tilde{s})\mu+\tilde{s}\cdot1,\rho)\\
 & =B(\tilde{s},\rho,1).
\end{alignat*}
Likewise, for any $r\in[0,\delta]$, we have $B(r,\rho,1)\in[0,\varepsilon\}]$
and so $g_{-1}(B(r,\rho,1),\rho)$ lies in the interval $[0,\mu(\rho)]$
and we find $\tilde{r}\geq0$ such that 
\[
g_{-1}(B(r,\rho,1),\rho)=(1-\tilde{r})\mu+\tilde{r}\cdot0.
\]
Taking $b(\cdot,\rho)$ of both sides, we obtain
\begin{alignat*}{1}
B(r,\rho,1) & =b((1-\tilde{r})\mu+\tilde{r}\cdot0,\rho)\\
 & =B(\tilde{r},\rho,0).
\end{alignat*}

For the second part, we have $B(0,\rho,x)=B_{s}(0,\rho,x)=0$ and
\begin{alignat*}{1}
B_{ss}(0,\rho,x) & =\frac{(x_{0}-\mu_{0})^{2}}{\mu_{0}}+\frac{(x_{1}-\mu_{1})^{2}}{\mu_{1}}\\
 & =\frac{((1-x)-(1-\mu)^{2}}{1-\mu}+\frac{(x-\mu)^{2}}{\mu}\\
 & =(x-\mu)^{2}\frac{1}{\mu(1-\mu)}\\
 & =\begin{cases}
e^{\rho} & \text{if }x=0\\
e^{-\rho} & \text{if }x=1
\end{cases}
\end{alignat*}
Then by Taylor's theorem centered around $0$, we have
\begin{alignat}{1}
B(s,\rho,0) & =\frac{e^{\rho}}{2}s^{2}+\frac{1}{6}B_{sss}(\xi,\rho,0)s^{3}\label{eq:taylor1}
\end{alignat}
for some $\xi\in(0,s)$ and 
\begin{alignat}{1}
B(\tilde{s},\rho,1) & =\frac{e^{-\rho}}{2}\tilde{s}^{2}+\frac{1}{6}B_{sss}(\tilde{\xi},\rho,1)\tilde{s}^{3}\label{eq:taylor2}
\end{alignat}
for some $\tilde{\xi}\in(0,\tilde{s})$. Combining (\ref{eq:equality}),
(\ref{eq:taylor1}), and (\ref{eq:taylor2}), we have 
\[
\left[\frac{e^{\rho}}{2}+\frac{1}{6}B_{sss}(\xi,\rho,0)s\right]s^{2}=\left[\frac{e^{-\rho}}{2}+\frac{1}{6}B_{sss}(\tilde{\xi},\rho,1)\tilde{s}\right]\tilde{s}^{2}
\]
and so
\[
\frac{\tilde{s}}{s}=\sqrt{\frac{\dfrac{e^{\rho}}{2}+\dfrac{1}{6}B_{sss}(\xi,\rho,0)s}{\dfrac{e^{-\rho}}{2}+\dfrac{1}{6}B_{sss}(\tilde{\xi},\rho,1)\tilde{s}}}.
\]
Taking the limit of both sides as $s\to0$, we obtain
\[
\frac{d\tilde{s}}{ds}(0,\rho)=\lim_{s\to0}\frac{\tilde{s}-0}{s-0}=e^{\rho}.
\]

Finally, either by repeating the above argument or observing $\tilde{r}(\cdot,\rho)$
and $\tilde{s}(\cdot,\rho)$ are inverse to each other, we have
\[
\frac{d\tilde{r}}{dr}(0,\rho)=e^{-\rho}
\]
and the second part of the lemma follows.
\end{proof}
Having proved the technical lemma above, we are ready to show the
limit of $D_{t}(B,\rho)B_{\rho}$ exists.
\begin{lem}
The function $(s,\rho,x)\mapsto D_{t}\bigl(B(s,\rho,x),\rho\bigr)B_{\rho}(s,\rho,x)$
can be extended continuously to $[0,1]\times\mathbb{R}\times\partial\Delta$.
In particular,
\[
\lim_{(s,\rho)\to(0,\tilde{\rho})}D_{t}\left(B(s,\rho,x),\rho\right)B_{\rho}(s,\rho,x)=0
\]
for all $\tilde{\rho}\in\mathbb{R}$ and for $x=0$ or $x=1$.\end{lem}
\begin{proof}
Recalling the definition of the facets $\partial^{i}\Delta$ and vertices
$v_{i}$ of the unit simplex $\Delta$ defined in Section (\ref{sec:Notation-and-Properties}),
we have $\partial^{0}\Delta=\{1\}$, $\partial^{1}\Delta=\{0\}$,
$v_{0}=0$, and $v_{1}=1$. Using (\ref{eq:Dt}), we have
\begin{alignat*}{1}
D_{t}\bigl(B(s,\rho,v_{k}),\rho\bigr)B_{\rho}(s,\rho,v_{k}) & =\sum_{i=0}^{1}\mu_{i}(\rho)h_{t}\bigl(B(s,\rho,v_{k}),\rho,\partial^{i}\Delta\bigr)B_{\rho}(s,\rho,v_{k}).\\
 & =\sum_{i=0}^{1}\mu_{i}(\rho)h_{t}\bigl(B(s,\rho,v_{k}),\rho,v_{1-i}\bigr)B_{\rho}(s,\rho,v_{k}).
\end{alignat*}
If $1-i=k$, then we have 
\begin{alignat*}{1}
h_{t}\bigl(B(s,\rho,v_{k}),\rho,v_{1-i}\bigr)B_{\rho}(s,\rho,v_{k}) & =h_{t}\bigl(B(s,\rho,v_{k}),\rho,v_{k}\bigr)B_{\rho}(s,\rho,v_{k})\\
 & =-h_{\rho}\bigl(B(s,\rho,v_{k}),\rho,v_{k}\bigr).
\end{alignat*}
If $1-i\neq k$, then $i=k$ and we have 
\begin{alignat*}{1}
h_{t}\bigl(B(s,\rho,v_{k}),\rho,v_{1-i}\bigr)B_{\rho}(s,\rho,v_{k}) & =\frac{1}{B_{s}\Bigl(h\bigl(B(s,\rho,v_{k}),\rho,v_{1-i}\bigr),\rho,v_{1-i}\Bigr)}B_{\rho}(s,\rho,v_{k})\\
 & =\frac{-B_{s}(s,\rho,v_{k})h_{\rho}\bigl(B(s,\rho,v_{k}),\rho,v_{k}\bigr)}{B_{s}\Bigl(h\bigl(B(\tilde{s},\rho,v_{1-k}),\rho,v_{1-i}\bigr),\rho,v_{1-i}\Bigr)}\\
 & =-\frac{B_{s}(s,\rho,v_{k})}{B_{s}(\tilde{s},\rho,v_{1-k})}h_{\rho}\bigl(B(s,\rho,v_{k}),\rho,v_{k}\bigr)\\
 & =-\frac{B_{s}(s,\rho,v_{k})}{B_{s}(\tilde{s},\rho,v_{1-k})}\frac{\frac{1}{s}}{\frac{1}{\tilde{s}}}\frac{1}{\frac{\tilde{s}}{s}}h_{\rho}\bigl(B(s,\rho,v_{k}),\rho,v_{k}\bigr)\\
 & =-\frac{B_{ss}(\xi,\rho,v_{k})}{B_{ss}(\eta,\rho,v_{1-k})}\frac{1}{\frac{d\tilde{s}}{ds}(\zeta,\rho)}h_{\rho}(B(s,\rho,v_{k}),\rho,v_{k}),
\end{alignat*}
where the second equality is given by Lemma \ref{lem:stilde}, and
in the last equality we have $\xi,\zeta\in(0,s)$ and $\eta\in(0,\tilde{s})$
by three separate applications of the mean value theorem.

We conclude that $h_{t}(B(s,\rho,v_{k}),\rho,v_{1-i})B_{\rho}(s,\rho,v_{k})$
goes to $0$ as $(s,\rho)$ goes to $(0,\tilde{\rho})$ for any $\tilde{\rho}\in\mathbb{R}$
regardless of whether $i=k$ or $1-i=k$ and hence $D_{t}(B,\rho)B_{\rho}$
goes to $0$ as $(s,\rho)$ goes to $(0,\tilde{\rho})$.
\end{proof}
Thus, we have shown that (\ref{eq:integrand}) is continuous on $[0,1]\times\mathbb{R}^{m}\times\partial\Delta$
for all $m\geq1$ and we may apply Lebesgue\textquoteright{}s dominated
convergence theorem to (\ref{eq:integralinstep4}). Hence, the function
(\ref{eq:T_i}) has a continuous partial derivative in $\rho_{q}$
for all $q$. Another application of Lebesgue's dominated convergence
theorem shows that the first summand (\ref{eq:firstSUMMANDrewrite})
has a continuous partial derivative in $\rho_{q}$ for all $q$ and
is given by
\[
\frac{\partial}{\partial\rho_{q}}\int_{0}^{\infty}n[1-D]^{n-1}Sdt=\int_{0}^{\infty}\frac{\partial}{\partial\rho_{q}}\left[n[1-D]^{n-1}S\right]dt.
\]

\subsubsection{Summary for $G$}

Since both the first and second summand of $G_{\rho_{p}}$ have continuous
partial derivatives in $\rho_{q}$ for all $q$, the function $G_{\rho_{p}}$
has a continuous partial derivative in $\rho_{q}$ for all $q$. Since
we could also differentiate under the integral sign of both summands,
the function $G$ has continuous second partial derivatives and 
\[
\frac{\partial}{\partial\rho_{q}}G_{\rho_{p}}=\int_{0}^{\infty}\frac{\partial}{\partial\rho_{q}}\left[n[1-D]^{n-1}D_{\rho_{p}}\right]dt,
\]
for all $p,q$.

\subsection{Conclusion of the proof}

The above concludes the proof of Theorem \ref{thm:maintheorem}. In
brief, we started with the function 
\[
F_{n}(\rho)=\int_{\Delta^{n}}\max_{j=1,\dots,n}\left[\bigl\langle\rho,\lambda^{j}\bigr\rangle-\bigl\langle\widehat{\lambda^{j}},\log\widehat{\lambda^{j}}\bigr\rangle\right]d\lambda^{1}\cdots d\lambda^{n}
\]
and rewrote it in Section \ref{sec:Additional-Preliminaries} as
\[
F_{n}(\rho)=\log(1+\left|e^{\rho}\right|)-\int_{0}^{\infty}[1-D(t,\rho)]^{n}dt.
\]
We then denoted the second term, without the minus sign, by $G_{n}$.
Working only with fixed $n$, we dropped the subscript $n$ and simply
wrote $G$. By Lemma \ref{lem:Drhop}, we concluded that both $D$
and $G$ are continuously differentiable in $\rho$. Furthermore,
we could bring the derivative under the integral sign so that
\[
\frac{\partial}{\partial\rho_{p}}G=\int_{0}^{\infty}n[1-D(t,\rho)]^{n-1}D_{\rho_{p}}(t,\rho)dt
\]
for any $p=1,\dotsc,m$. In Section \ref{sub:SecondPartialG}, we
separated the integral on the right-hand side into two terms and showed
that each term is continuously differentiable. Hence, we concluded
the function $G$ is twice differentiably continuous. Since $\log(1+\left|e^{\rho}\right|)$
is also $\mathcal{C}^{2}$, we conclude that $F_{n}$ is twice differentiably
continuous and the remaining statements of the theorem follow directly.

\bibliographystyle{abbrv}
\bibliography{ThesisBibliography}

\end{document}